\documentclass[preprint,12pt]{elsarticle}

\usepackage{amssymb,enumitem,amsmath, mathtools}
\usepackage{amsthm}
\usepackage{pgf,tikz,mathtools,mathrsfs,relsize,eqnarray}
\usepackage{tikz,bm}
\usepackage{hyperref}
\usepackage{ulem} 
\usepackage{xcolor}
\usepackage{float}
\usepackage{graphicx} 
\usepackage{tabularx}
\usepackage{caption} 
\usepackage{subcaption}
\usepackage[export]{adjustbox}
\usepackage{verbatim}
\usepackage{sidecap}
\usepackage{tikz}
\usepackage{algorithm}
\usepackage{algpseudocode}
\usepackage{eqparbox}
\usepackage{booktabs}
\usepackage{pbox}
\usepackage{comment}
\usepackage{cancel} 
\newdimen{\algindent}
\algnewcommand\LeftComment[2]{%
\hspace{#1\algindent}$\triangleright$ \eqparbox{COMMENT}{#2} \hfill %
}

\journal{ArXiv}

\newtheorem{theorem}{Theorem}[section]
\newtheorem{thm}{Theorem}[section]

\newtheorem{lemma}{Lemma}[section]

\newtheorem{remark}{Remark}[section]
\newtheorem{definition}{Definition}[section]
\newtheorem{alg}[thm]{Algorithm}

\newcommand{\Grad}{\ensuremath{\nabla}}
\newcommand{\omint}{\int_{\Omega}}

\newcommand{\pdv}[2][]{\frac{\partial#1}{\partial#2}}

\newcommand{\pare}[1]{\left({}#1\right)}
\newcommand{\lip}[2]{\left({}#1,#2\right){}}

\newcommand{\eps}{\varepsilon}

\newcommand{\delt}{\ensuremath{\Delta t}}

\newcommand{\hrho}{\hat{\rho}}
\newcommand{\ph}{{\rho}_h}
\newcommand{\vh}{v_h}

\newcommand{\R}{\mathbb R}

\newcommand{\ohm}{\Omega}

\newcommand{\E}{\mathcal{E}}
\newcommand{\Z}{\mathcal{Z}}

\newcommand{\be}{\begin{equation}}
\newcommand{\ee}{\end{equation}}
\newcommand{\bea}{\begin{eqnarray}}
\newcommand{\eea}{\end{eqnarray}}
\newcommand{\beas}{\begin{eqnarray*}}
\newcommand{\eeas}{\end{eqnarray*}}
\newcommand{\ba}{\begin{array}}
\newcommand{\ea}{\end{array}}

\newcommand{\norm}[1]{\ensuremath{\left\|{#1}\right\|}}
\newcommand{\trinorm}[1]{{\left\vert\kern-0.15ex\left\vert\kern-0.15ex\left\vert #1 
    \right\vert\kern-0.15ex\right\vert\kern-0.15ex\right\vert}}
\newcommand{\abs}[1]{\ensuremath{\left\lvert{#1}\right\rvert}}

\long\def\bcom#1\ecom{}

\begin{document}

\begin{frontmatter}



\title{Numerical Analysis of a Bio-Polymerization Model }

\author[UNLV]{Ali Balooch}
\ead{ali.balooch@unlv.edu}
\author[Penn]{Faranak Courtney-Pahlevani}
\ead{fxp10@psu.edu}
\author[MS]{Lisa Davis}
\ead{lisa.davis@montana.edu}
\author[UPB]{Adrian  Dunca}
\ead{argus_adrian.dunca@upb.ro}
\author[UNLV]{Monika Neda}
\ead{monika.neda@unlv.edu}
\ead[url]{http://neda.faculty.unlv.edu/}
\author[VT]{Jorge Reyes\corref{cor1}}
\ead{reyesj@vt.edu}

\cortext[cor1]{Corresponding author}

\affiliation[UNLV]{organization={Department of Mathematical Sciences},
            addressline={University of Nevada, Las Vegas}, 
            city={Las Vegas},
            postcode={Box 454020}, 
            state={NV},
            country={USA}}
            
\affiliation[Penn]{organization={Division of Science \& Engineering },
            addressline ={ Penn State University - Abington }, 
            city={ Abington },
            postcode={19001}, 
            state={PA},
            country={USA}}
            
\affiliation[MS]{organization={Department of Mathematical Sciences},
            addressline={Montana State University}, 
            city={Bozeman},
            postcode={Box 172400}, 
            state={MT},
            country={USA}}

\affiliation[UPB]{organization={Department of Mathematics and Informatics},
            addressline={Politehnica University of Bucharest}, 
            city={Bucharest},
            country={Romania}}
            
\affiliation[VT]{organization={Department of Mathematics},
            addressline ={ Virginia Tech}, 
            city={ Blacksburg},
            postcode={24060}, 
            state={VA},
            country={USA}}
            
\begin{abstract}

This work studies a stabilization technique for first-order hyperbolic differential equations used in DNA transcription modeling. Specifically we use the Lighthill-Whitham-Richards Model with a nonlinear Greenshield's velocity proposed in \cite{davis2024numerical}. Standard finite element methods are known to produce spurious oscillations when applied to nonsmooth solutions. To address this, we incorporate stabilization terms involving spatial and temporal filtering into the system. We present numerical stability and prove convergence results for both the backwards Euler and time filtered formulations. We also present several computational results to demonstrate the rates in space and in time as well as for selected biological scenarios.

\end{abstract}



\end{frontmatter}

\section{Introduction}
\label{sec:intro}
This paper describes the analysis of a simulation framework for a comprehensive model of ribosome abundance control in bacteria.  The full model combines a mathematical description of bio-polymerization processes: \textit{transcription} of ribosomal RNA, transcription of \textit{messenger RNA } (mRNA) of ribosomal proteins (r-proteins) and \textit{translation} of r-proteins with a characterization of feedback mechanisms that regulate initiation rates, processing rates and abundance of key molecules.  Here we focus on the prototype for each of the compartment models describing the processes of transcription and translation.   These processes are fundamentally characterized by the motion of a molecular motor copying a segment of DNA or mRNA.  The speed at which these motors travel along the strand while reading the gene is referred to as the \textit{elongation velocity}, and in the authors' previous work \cite{davis2021accurate,davis2014_2,davis2024numerical}, the LWR model with the Greenshield's velocity is used for each compartment.  The model is useful for describing the situation where many motors are copying the segment simultaneously.  In such cases, the elongation velocities can be non-uniform, resulting in significant variations in the density of the molecular motors in different spatial regions of the DNA segment.  For a comprehensive description of the model development, the reader is referred to \cite{davis2024numerical} and the references therein.

The LWR with the Greenshield's model \cite{Greenshields_1935} is given by 
\begin{equation}
\frac{\partial \rho}{\partial t} + \left ( v_{f}- \frac{2v_{f}}{\rho_{m}}\rho \right )\frac{\partial \rho}{\partial x} = 0,
\label{eq: problem_description}
\end{equation}
where $v_{f}$ is the free flow speed and $\rho_{m}$ is the maximum jam density. The free flow speed $v_{f}$ represents the speed of the traffic when the density $\rho$ is zero. The maximum density $\rho_{m}$ is the traffic density at which the speed of traffic $v$ is equal to zero. 

It is known that when we solve the hyperbolic partial differential equation (PDE) given by equation \eqref{eq: problem_description}, oscillations exist around the shock solutions.  Herein, we apply a stabilization of finite element method (FEM) introduced in \cite{dunca2015vreman} for the advection equation based on Vreman filtering \cite{vreman2003filtering}. Additionally, we use a simple time filter 
first introduced in \cite{guzel2018time}. This time filter does not add computational complexity and is shown to be second order accurate in time \cite{guzel2018time,decaria2020time}.

In \cite{BDPS23} the time filter is shown to be effective at improving accuracy when combined with an explicit first order upwind finite difference method with minimal expense in numerical implementation for the linear advection equation case.  There the filtered scheme introduces a small amount of dissipation into the behavior of a non-smooth solution, thereby increasing accuracy while preventing spurious oscillations. 

The remainder of the paper is organized as follows. Section 2 has notation, preliminaries, and algorithms. Section 3 presents numerical results for the stability and convergence analysis of the model, followed by computational experiments in Section 4. We conclude our work in Section 5.

\section{Notation, Preliminaries and Algorithms}
\label{sec:Notation}
Let the interval $\Omega \subset \R$ denote the  domain, while the $L^{2}(\Omega)$ norm and the inner product are denoted by $\left \| \cdot \right \|$ and $ (\cdot, \cdot) $ respectively. The norm of the $ H^k(\ohm)$ space is denoted by $ \norm{\cdot}_k$.
The analysis is carried out in the periodic setting. Thus, the function space used is $X:=H_{\#}^{1}(\Omega)$, which is the closure of the $C^{\infty}$ periodic functions in $ \Omega$ in the $ H^1$ norm.
Let $ X_h \subset X $ be the subspace of finite elements.

The continuous mean $ \overline{u} \in X$ of $ u \in X $ is the solution of the PDE \cite{germano1986differential},
\begin{equation*}
    -\delta^2 \Delta \overline{u}+ \overline{u} = u.
\end{equation*}
We denote $ G:L^2(\ohm) \to X$ as the filtering operator, i.e. $ G(u) = \overline{u}$.
Furthermore, we consider the discrete mean $ \overline{u}^h \in X_h $ of $ u \in X$, defined as 
the unique solution of
\begin{equation*}\label{filterdfn}
\delta^{2}(\nabla\overline{u}^{h},\nabla v_{h})+(\overline{u}^{h}%
,v_{h})=(u,v_{h}), \ \ \forall v_{h}\in X_{h},
\end{equation*}
where the mean quantities are associated with the filtering length scale $\delta$ \cite{manica2007finite}. We denote the discrete filtering operator as $ G_h: L^2(\Omega) \to X_h $, where $ G_h(u) = \overline{u}^h $. Next, we define the deconvolution operators,

\begin{definition}
\label{defvC} The $N^{th}$ order van Cittert continuous and discrete
deconvolution operators as $D_{N}: X \mapsto X$ and $D_{N}^{h}:X_ \mapsto X_h$ are, respectively,
\begin{equation*}
D_{N}\ :=\ \sum_{n=0}^{N}(I\,-\,G)^{n}\ ,\ \ \mbox{and }\ \ D_{N}^{h}%
\ :=\ \sum_{n=0}^{N}(I\,-\,G_{h})^{n}\,.
\end{equation*}
\end{definition}

The stability and accuracy properties of continuous and discrete averaging and deconvolution have been studied in \cite{BIL06,dunca2006stolz,manica2007finite,layton2008numerical,layton2012approximate,layton2008helicity}. The operators denoted as $\overline{\cdot}$ and $\overline{\cdot}^h$ represent filtering and lend the physical meaning to the parameter $\delta $ as the filter width. In general, in the finite element setting, $\delta$ depends on the mesh size denoted by $h$. %

We present several properties and lemmas used in our analysis below.

\begin{lemma}[Smoothing Property \cite{connors2010accuracy}]
\label{lem: Smoothing_lemma}
For any $ u_h \in X_h$,
\begin{equation*}
    \delta^2 \norm{ \triangle^h \pare{ D_N^h \overline{u_h}^h } } + \delta \norm{ \Grad \pare{ D_N^h \overline{u_h}^h } } + \norm{ D_N^h \overline{u_h}^h } \leq C(N) \norm{u_h},
\end{equation*}
where $  \triangle^h $ is the discrete Laplacian operator and $ C(N)$ is a generic constant that depends on the order of deconvolution $N$.
\end{lemma}

\begin{lemma}[Deconvolution error estimate \cite{dunca2015vreman}]\label{lem: deconvolution error}
There exist general constants $ C,C(N)$ such that for all $ u\in X\cap H^{2N+2}(\ohm) \cap H^{k+1}$,\\

$ \delta \norm{ \Grad u - \Grad D_N^h \overline{u}^h} \leq C(N) ( \delta h^k + h^{k+1} ) \norm{u}_{k+1} + C \delta^{2N+3} \norm{\Grad u}_{2N+3}. $
\end{lemma}
\noindent Additionally, we assume the following approximation properties, \cite{brenner2008mathematical,gunzburger2012finite}:
\bea
    \inf_{v \in X_{h}} \| u - v \| &\le& C h^{k+1} \| u \|_{k+1},\;\; u \in
    H^{k+1}(\Omega), \label{eq: l2projbound} \\
    \inf_{v\in X_{h}}\|\Grad( u - v) \| &\le& C h^{k} \| u \|_{k+1},\;\; u \in
    H^{k+1}(\Omega).  \label{eq: h1projbound}
\eea

\noindent There exists an interpolant $I_h$ of $u$ as constructed by Brenner and Scott \cite{brenner2008mathematical} so that \eqref{eq: l2projbound} and \eqref{eq: h1projbound} hold with $I_h$ instead of $v$. We will apply the approximation property 4.4.25 with $p=\infty, m=2,s=1,l=0,n=1$ on page 110, where $l$ is defined as in Lemma 4.4.41 in \cite{brenner2008mathematical}. Thus, $I_h$ approximates $u$ in the norm of the Sobolev space $W^{1}_{\infty}$ in the order of $h$. Due to these approximation properties of the interpolant $I_h$, it follows that if $\norm{\frac{\partial^2 u}{\partial x^2}}_{L^{\infty}}<C$ then both terms below are bounded uniformly
for $h<1$, i.e.
\begin{equation*}\label{regu2}
\norm{\frac{\partial I_h}{\partial x}}_{L^{\infty}} ,\norm{\frac{\partial ( u - I_h) }{\partial x}}_{L^{\infty}} <C.
\end{equation*}

Lemma \ref{lem: algebra_lemma} is a convenient algebraic identity from \cite{decaria2020time} which will be used in a similar manner for the convergence analysis.

\begin{lemma}
\label{lem: algebra_lemma}
The following identity holds
\beas
& \pare{\frac{3}{2}a-2b+\frac{1}{2}c}\pare{\frac{3}{2}a-b+\frac{1}{2}c} = \\
& \frac{1}{4}\pare{a^2+(2a-b)^2+(a-b)^2}- \frac{1}{4}\pare{b^2 +(2b-c)^2 +(b-c)^2 } + \frac{3}{4} (a-2b+c)^2.
\eeas
\end{lemma}

The convergence analysis uses a discrete Gronwall inequality \cite{heywood1990finite} stated below.
\begin{lemma}\label{lem:discreteGronwall} Let $\Delta t$, H, and $a_{n},b_{n},c_{n},\gamma_{n}$
(for integers $n \ge 0$) be finite nonnegative numbers such that
\begin{equation*}
a_{l}+\Delta t \sum_{n=0}^{l} b_{n} \le \Delta t \sum_{n=0}^{l} \gamma_{n}a_{n} +
\Delta t\sum_{n=0}^{l}c_{n} + H \ \ for \ \ l\ge 0. \label{gronwall1}
\end{equation*}
Suppose that $\Delta t \gamma_n < 1 \; \forall n$. Then,
\begin{equation*}
a_{l}+ \Delta t\sum_{n=0}^{l}b_{n} \le \exp\left( \Delta t\sum_{n=0}^{l} \frac{\gamma_{n}}{1 - \Delta t \gamma_n } \right) \left( \Delta t\sum_{n=0}^{l}c_{n} + H
\right)\ \ for \ \ l \ge 0.
\end{equation*}
\end{lemma}

Next we introduce the variational formulation for the LWR model with Greenshield's velocity and Vreman stabilization term. Find $\rho_h \in X_h$ satisfying for all $v \in X_{h}$
\begin{align*}
\lip{ \pdv{t} \rho_{h}}{ v_h}+v_{f}\lip{\pdv{x}{\rho_{h}}}{ v_h}-\frac{2v_{f}}{\rho_{m}} b(\rho_h,\rho_h,v_h)
+ \chi \delta^2 \lip{\pdv{x} \rho_{h}^* }{ \pdv{x} v_h^*}  = 0, 
\end{align*}

where the stabilization term $\chi \delta^2 \lip{\pdv{x} \rho^{*}}{ \pdv{x} v_h^*}  $, initially introduced by Vreman \cite{vreman2003filtering} is added. Here $ \rho^* = \rho - D_N^h\overline{\rho}^h$ and the goal of this term is to improve accuracy and damp unwanted spurious oscillations. 
This also introduces the dimensionless $ \chi$ as a stabilization parameter which can be manually tuned to increase or decrease the effects of the added stabilization term. 
Based on \cite{kang1997nonlinear} we also define trilinear term $ b(\cdot,\cdot,\cdot): X\times X \times X \to \mathbb{R} $ as 
\begin{equation}\label{eq: bterm}
     b(u,v,w) = \frac{1}{3} \omint \pare{\pdv{x}(uv) + u\pdv[v]{x} }w \; dx .
\end{equation}
In \cite{kang1997nonlinear,shen1995nonlinear} it is proven that $ b$ is skew-symmetric, i.e.
\begin{equation}\label{eq:skew_symm}
    b(u,v,w) + b(u,w,v) = 0.
\end{equation}

We also use the following lemma.
\begin{lemma} \label{lem: B_terms}
For $u, v, w \in X$ we have the following identities
    \bea
    \label{eq: f2} b(u,v,w)&=&\frac{1}{3}\int_{\Omega} u\frac{\partial v}{\partial x}w{\rm d}x-\frac{1}{3}\int_{\Omega} v\frac{\partial w}{\partial x}u{\rm d}x,\\
    \label{eq: f3} b(u,v,w)&=&-\frac{1}{3}\int_{\Omega} v\frac{\partial u}{\partial x}w{\rm d}x-\frac{2}{3}\int_{\Omega} v\frac{\partial w}{\partial x}u{\rm d}x, \\
    \label{eq: f4} 2\int_{\Omega} u\frac{\partial v}{\partial x}v{\rm d}x &=& -\int_{\Omega} v\frac{\partial u}{\partial x}v{\rm d}x.
    \eea    
\end{lemma}
\begin{proof}
    From \eqref{eq: bterm} we have that

    \begin{equation}\label{eq: f1} 
        b(u,v,w)=\frac{2}{3}\int_{\Omega} u\frac{\partial v}{\partial x}w{\rm d}x+\frac{1}{3}\int_{\Omega} v\frac{\partial u}{\partial x}w{\rm d}x.
    \end{equation}
Applying integration by parts on the second term results in \eqref{eq: f1} becoming
    \begin{align*}\label{f2}
        b(u,v,w) &= \frac{2}{3}\int_{\Omega} u\frac{\partial v}{\partial x}w{\rm d}x- \frac{1}{3}\int_{\Omega} u\frac{\partial }{\partial x}(v w){\rm d}x \nonumber\\
        &= \frac{1}{3}\int_{\Omega} u\frac{\partial v}{\partial x}w{\rm d}x-\frac{1}{3}\int_{\Omega} v\frac{\partial w}{\partial x}u{\rm d}x.
    \end{align*}
To prove \eqref{eq: f3} we just apply integration by parts on the first term of \eqref{eq: f1} instead giving
    \beas
     b(u,v,w) &=& -\frac{2}{3}\int_{\Omega} v\frac{\partial }{\partial x}(uw){\rm d}x+\frac{1}{3}\int_{\Omega} v\frac{\partial u}{\partial x}w{\rm d}x \nonumber\\
    &=& -\frac{1}{3}\int_{\Omega} v\frac{\partial u}{\partial x}w{\rm d}x-\frac{2}{3}\int_{\Omega} v\frac{\partial w}{\partial x}u{\rm d}x.
    \eeas
Lastly, to prove \eqref{eq: f4} consider
    \begin{equation*}
    2\int_{\Omega} u\frac{\partial v}{\partial x}v{\rm d}x= \int_{\Omega} u \frac{\partial }{\partial x}(v^2) {\rm d}x =-\int_{\Omega} v\frac{\partial u}{\partial x}v{\rm d}x 
    \end{equation*}
    
\end{proof}

Lastly, let $\delt$ denote the time step, then $t^{n}
= n \delt$, $n = 0, 1, \dots, M$, and final time is $T := M \delt$. Here, and in the sequel we adopt the notation $ u(t^n,\cdot) = u^n$
\begin{eqnarray*}
 \trinorm{ u }_{\infty, k} := \max_{0 \le n \le M} \|u^{n}\|_{k},
   \hspace{1 cm}
    \trinorm{ u }_{m,k} := \left(\delt\sum_{n=0}^{M} \| u^{n} \|^{m}_{k}\right)^{1/m} .
\end{eqnarray*}

\subsection{Algorithms}
We study two algorithms. First we employ the backward Euler temporal discretization which gives the following algorithm.

\begin{alg}
\label{alg: BEGM}
For $n = 1, 2 \ldots, M$ find $\rho_{h}^{n} \in X_{h}$ such that,
\bea\label{eq: BE_alg}
\frac{1}{\delt}\lip{\rho_{h}^{n} - \rho_{h}^{n-1}}{v_h} + v_f \pare{ \pdv{x} \rho_{h}^{n},v_h} - \frac{2v_f}{\rho_m} b(\rho_{h}^{n},\rho_{h}^{n},v_h) \nonumber \\
+ \chi \delta^2 \pare{ \pdv{x} \rho_{h}^{{n}^*}, \pdv{x} v_h^* }  = 0, \: \forall v_h \in X_{h}.
\eea

\end{alg}

We also implement the time filter studied in \cite{guzel2018time,decaria2020time,DPSR_2021} as a post-processing step, which yields the following algorithm.

\begin{alg}
\label{alg: FDTFGM}
For $n = 2, 3 \ldots, M$ find $\rho_{h}^{n} \in X_{h}$ such that,\\
{\it{Step 1: \textcolor{black}{Backward Euler}}}\\
\bea\label{eq: step1_BE}
\frac{1}{\delt}\lip{\hrho_{h}^{n} - \rho_{h}^{n-1}}{v_h} + v_f \pare{ \pdv{x} \hrho_{h}^{n},v_h} - \frac{2v_f}{\rho_m} b(\hrho_{h}^{n},\hrho_{h}^{n},v_h) \nonumber \\
+ \chi \delta^2 \pare{ \pdv{x} \hrho_{h}^{{n}^*}, \pdv{x} v_h^* }  = 0, \: \forall v_h \in X_{h},
\eea
{\it{Step 2: \textcolor{black}{Time Filter}}}\\
\beas\label{eq: step2_TF}
\rho_{h}^n  = \hrho_{h}^n - \frac{\gamma}{2} \pare{ \hrho_{h}^n - 2\rho_{h}^{n-1} + \rho_{h}^{n-2} }.
\eeas
\end{alg}

Algorithm \ref{alg: FDTFGM} approximates the density $\rho^n$ from the intermediate density approximation from step 1, i.e. $\hrho^n$. \textcolor{black}{ This increases the convergence rate in time from first to second order.} 

By choosing $ \gamma = \frac{2}{3}$  in the time filter equation in step 2 we obtain for $ \hrho_h^n$ the following,
\bea
\label{eq:time_filter}
\hrho_{h}^n = \frac{3}{2} \rho_{h}^{n} - \rho_{h}^{n-1} +\frac{1}{2} \rho_{h}^{n-2}.
\eea
Replacing $\hrho_h^n$ in step 1 by (\ref{eq:time_filter}), one can reduce the two step process  to a single equation. For simplicity in notation, we introduce the interpolation and difference operators as 
\bea
I[u^{n}] &=&  \frac{3}{2} u^{n} - u^{n-1} +\frac{1}{2} u^{n-2}, \label{eq:interpolation}
\text{ and} \\
D[u^{n}] &=&  \frac{3}{2} u^{n} - 2 u^{n-1} +\frac{1}{2} u^{n-2},
\label{eq:diffrence}
\eea
as well as 
\bea
\mathcal{E}[u^n] &=& \frac{1}{4} \pare{ \norm{u^n}^2  + \norm{2u^n - u^{n-1} }^2 + \norm{ u^n - u^{n-1} }^2 }, \label{eq: epsDef} \\
\mathcal{Z}[u^n] &=& \frac{3}{4} \norm{ u^{n} -u^{n-1} - u^{n-2} }^2 \label{eq: zetadef},
\eea
which are used in the analysis of Algorithm \ref{alg: FDTFGM}.

Incorporating \eqref{eq:interpolation} and \eqref{eq:diffrence} into Algorithm \ref{alg: FDTFGM} and using (\ref{eq:time_filter}), we obtain the following equivalent scheme
\bea
\frac{1}{\delt} \lip{D[\rho_h^n]}{v_h} + v_f \pare{ \pdv{x} I[\rho_{h}^{n}],v_h} - \frac{2v_f}{\rho_m} b(I[\rho_{h}^{n}],I[\rho_{h}^{n}],v_h) \nonumber \\
+ \chi \delta^2 \pare{ \pdv{x} I[\rho_{h}^{n}]^* , \pdv{x} v_h^*  }  = 0. \:
\label{eq: FDTFFEM}
\eea
The initial condition $ \rho_h^0 $ is the $ L^2$ projection of $ \rho_0$ and $ \rho_h^1$ is obtained using backward Euler. 

\begin{lemma}[ \cite{decaria2020time} ]
    For $u$ sufficiently smooth, we have 
    \bea
        \norm{ \frac{D[u(t^{n+1})]}{\delt} - u_t(t^{n+1}) }^2 &\leq& \frac{6}{5} \delt^3 \int_{t^{n-1}}^{t^{n+1}} \norm{u_{ttt}}^2 \; dt, \label{eq: D_TF}\\
        \norm{ I[u(t^{n+1})] - u(t^{n+1})}^2 &\leq& \frac{4}{3} \delt^3 \int_{t^{n-1}}^{t^{n+1}} \norm{u_{tt}}^2 \; dt.
    \eea
\end{lemma}

\section{Numerical Analysis}
\label{sec: main}
First we present the numerical stability of our algorithms followed by convergence theorems.

\begin{lemma}
\label{lem: Filter Stability Lemma}
Solutions to the fully discrete Algorithm \ref{alg: BEGM} are unconditionally stable and satisfy
\beas
 \norm{\rho_h^{{M}}}^2 +
 2\chi\delta^2 \delt \sum_{n=1}^{{M}} \norm{ \pdv{x} \rho_h^{n^*} }^2 &\leq& { \norm{\rho_0}^2} .
\eeas
\end{lemma}
\begin{proof}
    Setting $v_h=\rho_h^n$ in (\ref{eq: BE_alg}) yields
        \bea
            \frac{1}{\Delta t}\norm{\rho_h^n}^2-\frac{1}{\Delta t}(\rho_h^{n-1},\rho_h^n)+\chi\delta^2\norm{\frac{\partial}{\partial x}\rho_h^{n^*}}^2=-v_f\pare{\frac{\partial}{\partial x}\rho_h^n,\rho_h^n}+\frac{2v_f}{\rho_m}b(\rho_h^n,\rho_h^n,\rho_h^n).\label{eq:var_form_rho}
        \eea
    The first term on the right-hand side vanishes due to the periodicity of functions in $X_h$, as
        \beas
            \pare{\frac{\partial}{\partial x}\rho_h^n,\rho_h^n}=\frac{1}{2}\int_\Omega\frac{\partial}{\partial x}(\rho_h^n)^2\,dx=0.
        \eeas
    Since the trilinear term $b(\cdot,\cdot,\cdot)$ is skew-symmetric by (\ref{eq:skew_symm}), the other term on the right-hand side also vanishes. Note that applying Cauchy-Schwarz inequality and Young's inequality in succession gives
        \[
            -\frac{1}{\delt}(\rho_h^{n-1},\rho_h^n)\ge-\frac{1}{2\delt}\norm{\rho_h^{n-1}}^2-\frac{1}{2\delt}\norm{\rho_h^n}^2.
        \]
    Then (\ref{eq:var_form_rho}) becomes
        \beas
            \frac{1}{2\delt}\norm{\rho_h^n}^2-\frac{1}{2\delt}\norm{\rho^{n-1}}^2+\chi\delta^2\norm{\frac{\partial}{\partial x}\rho_h^{n^*}}^2\le0.
        \eeas
    Multiplying by $2\delt$ and summing $n=1,2,\ldots, M$ results in
        \beas
            \norm{\rho_h^M}^2-\norm{\rho_h^0}^2+2\chi\delta^2\delt\sum_{n=1}^{{M}} \norm{ \pdv{x} \rho_h^{n^*} }^2\le0,
        \eeas
    and recalling that $\rho_h^0$ is the $L^2$ projection of $\rho_0$ completes the proof.
\end{proof}

From \cite{davis2024numerical} we have the following stability result for Algorithm \ref{alg: FDTFGM} with $ \gamma = \frac23 $.

\begin{lemma}
\label{Filter Stability Lemma}
The solution to Algorithm \ref{alg: FDTFGM} given by \eqref{eq: FDTFFEM} is unconditionally stable and it satisfies

\bea
 \norm{\rho_h^M}^2 + 4\sum_{n=2}^{M}\mathcal{Z}[\rho_h^n] + 2\chi\delta^2 \delt \sum_{n=2}^{M} \norm{ \pdv{x} I[\rho_h^n]^* }^2 &\leq&  C ( \rho_0 , \rho_1 ),
 \label{eq: TF_stability}
\eea
where $ \Z[\rho_h^n]$ is defined by \eqref{eq: zetadef}.
\end{lemma}

Next, we will study the convergence analysis.

\begin{theorem}\label{thm:ErrorEstimate}
Let $ \ph^n $ be the solution of Algorithm \ref{alg: BEGM} and $\rho$ be a solution of \eqref{eq: problem_description} with enough regularity. Then, for sufficiently small $\delt $ we have
\beas
    && \trinorm{\rho^{M}- \ph^{M}}_{\infty,0}^2   \leq  C h^{2k+2}\trinorm{\rho}_{\infty,k+1}+ CK \bigg(\chi \delta^{4N+6} \trinorm{\pdv[\rho]{x}}^2_{2N+3}  \nonumber\\
    && +\chi h^{2k+2} \trinorm{ \rho }_{2,k+1}^2+ \chi\delta^2 h^{2k} \trinorm{ \rho}_{2,k+1}^2\nonumber \\
    &&\hspace*{-.2cm}+ \frac{v_f^2}{\delta^2} \pare{\frac{1}{\chi} +1 } h^{2k+2} \trinorm{\rho}_{2,k+1}^2 + h^{2k+2} \int_{t^0}^{t^M} \norm{\rho_t} _{k+1}^2 {\rm d}t \nonumber \\
    &&\hspace*{-.2cm}+ \delt^2 \int_{t^0}^{t^M} \norm{ \rho_{tt} }^2 \; dt 
    + \frac{v_f^2}{\rho_m^2} \pare{ 1+ \frac{1}{\chi\delta^2}{+\frac{1}{\delta^2}} } h^{2k+2} \trinorm{\rho}_{2,k+1}^2 \bigg),
\eeas
where $K = \exp\pare{ \delt \sum_{n=0}^{M} \frac{ \gamma_n }{1-\delt \gamma_n } }$ and $ \gamma_n = C\pare{ \frac{v_f}{\rho_m} +1 } $. 

\end{theorem}

\begin{proof}
Evaluating the weak formulation of \eqref{eq: problem_description} at time $ t^n$ yields
\beas
   \frac{1}{\delt} \pare{\rho^n - \rho^{n-1},v_h } +  v_{f}\pare{\pdv{x} \rho^n , v_h}  - \frac{2v_f}{\rho_m} b(\rho^{n},\rho^n,v_h)
+ \chi \delta^2 \pare{ \pdv{x} \rho^{n^*} , \pdv{x} v_h^*}\nonumber \\
= \pare{ \frac{\rho^n - \rho^{n-1}}{\delt} - \pdv{t} \rho^n , v_h} + \chi \delta^2 \pare{ \pdv{x} \rho^{n^*} , \pdv{x} v_h^*}.
\eeas
Subtracting \eqref{eq: BE_alg} from the above equation gives
\beas
  \frac{1}{\delt} \pare{e^n - e^{n-1},v_h } +  v_{f}\pare{ \pdv{x} e^n , v_h} - \frac{2v_f}{\rho_m} \pare{ b(\rho^n,\rho^n,v_h) - b(\ph^n,\ph^n, v_h) } \nonumber \\
  + \chi \delta^2 \pare{ \pdv{x} e^{n^*}, \pdv{x} v_h^* } = \pare{ \frac{\rho^n - \rho^{n-1}}{\delt} - \pdv{t} \rho^n, v_h} + \chi \delta^2 \pare{ \pdv{x} \rho^{n^*} , \pdv{x} v_h^*},
\eeas
 where $ e^n = \rho^n - \rho_h^n$. Rewriting the trilinear terms by adding and subtracting  $b(\ph^n,\rho^n,v_h)$ yields
\beas
  \frac{1}{\delt} \pare{e^n - e^{n-1},v_h } +  v_{f}\pare{ \pdv{x} e^n , v_h} - \frac{2v_f}{\rho_m} \pare{ b(e^n,\rho^n,v_h) + b(\ph^n,e^n, v_h) } \nonumber \\
  + \chi \delta^2 \pare{ \pdv{x} e^{n^*}, \pdv{x} v_h^* } = \pare{ \frac{\rho^n - \rho^{n-1}}{\delt} - \pdv{t} \rho^n, v_h} + \chi \delta^2 \pare{ \pdv{x} \rho^{n^*} , \pdv{x} v_h^*}.
\eeas
Next, we decompose the error as $ e^{n} = \rho^n - I_h^n+ I_h^n- \ph^n  = \eta^{n} + \phi^{n}_{h}$, where $ I_h^n$ is the interpolant of $ \rho^n $ in $ X_h$ as defined in Section \ref{sec:Notation}.  We set $v_h=\phi_h^n$ and use $ b(\ph^n,\phi_h^n,\phi_h^n) = 0 $ to obtain
\bea
  & & \frac{1}{2\delt} \pare{ \norm{\phi_h^n}^2 - \norm{\phi_h^{n-1}}^2 } + \chi \delta^2 \norm{ \pdv{x} \phi_h^{n^*}}^2= \nonumber \\
  & & - \pare{{\frac{\eta^n - \eta^{n-1}}{\delt}}, \phi_h^n} -v_{f}\pare{ \pdv{x} \phi_h^n , \phi_h^n}- v_{f}\pare{ \pdv{x} \eta^n , \phi_h^n} \nonumber\\
  & & + \frac{2v_f}{\rho_m} \pare{ b(\eta^n,\rho^n,\phi_h^n) + b(\phi_h^n,\rho^n, \phi_h^n) + b(\ph^n,\eta^n,\phi_h^n) }- \chi \delta^2 \pare{ \pdv{x} \eta^{n^*}, \pdv{x} \phi_h^{n^*} } 
    \nonumber \\
 & & + \pare{ \frac{\rho^n - \rho^{n-1}}{\delt} - \pdv{t} \rho^n, \phi^n_h} + \chi \delta^2 \pare{ \pdv{x} \rho^{n^*} , \pdv{x} \phi_h^{n^*}} \nonumber\\
 && \leq|T_0|+ |T_1| + |T_2| + |T_3| + |T_4| + |T_5| +
  |T_6| + |T_7| + |T_8|. \label{eq: error1}
\eea
Now we bound each $T_i$ term individually. $ T_0$ is bounded as 
\bea\label{eq:t0_bound}
    |T_0| &\leq&\frac{1}{4}\norm{\frac{\eta^n-\eta^{n-1}}{\delt}}^2+\norm{\phi_h^n}^2\nonumber\\
    &\leq&\frac{1}{4}\int_\Omega\pare{\frac{1}{\delt}\int_{t^{n-1}}^{t^n}|\eta_t|dt}^2dx+\norm{\phi_h^n}^2\nonumber\\
    &\leq&\frac{1}{4\delt} \int_{t^{n-1}}^{t^n} \norm{\eta_t}^2{\rm d}t + \norm{\phi_h^n}^2.
\eea
$ T_1$ vanishes by periodicity of $\phi_h^n\in X_h$ and the fundamental theorem of Calculus   
\bea
& |T_1|& = v_f \abs{\lip{\pdv{x}\phi_{h}^n}{ \phi_{h}^n} }
=  v_f  \abs{\int_{\Omega} \frac{\partial \phi_h^n}{\partial x} \phi_h^n \, dx} \leq  v_f \int_{\Omega} \frac12 \pdv{x}  |\phi_h^n|^2 \, dx = 0 . \qquad
\label{eq: t1_bound}
\eea
Using integration by parts and $ \phi_h = \phi_h^* + D_N^h \overline{\phi_h}^h $ along with Lemma \ref{lem: Smoothing_lemma} we have
\bea\label{t2_bound}
|T_2| &\leq& v_f \abs{\pare{ \eta^n, \pdv{x} \phi_h^n } } \leq v_f \abs{\pare{ \eta^n,\pdv{x}\phi_h^{n^*} }}+ v_f \abs{\pare{ \eta^n,\pdv{x} D_N^h \overline{\phi_h^n}^h }} \nonumber \\
 &\leq& \frac{Cv_f^2}{\chi\delta^2} \norm{\eta^n}^2 + \frac{\chi\delta^2}{12} \norm{\pdv{x} \phi_h^{n^*}}^2 + v_f \norm{\eta^n} \norm{\pdv{x} D_N^h \overline{\phi_h^n}^h } \nonumber\\ 
&\leq&  \frac{Cv_f^2}{\chi\delta^2} \norm{\eta^n}^2 + \frac{\chi\delta^2}{12} \norm{\pdv{x} \phi_h^{n^*}}^2 + v_f^2 \frac{C^2(N)}{\delta^2} \norm{\eta^n}^2 + \frac{1}{4}\norm{ \phi_h^n }^2.
\eea
To bound $|T_3|$ we use Lemma \ref{lem: B_terms} equation \eqref{eq: f2}

\bea\label{eq: t3_1}
|T_3|=\frac{2v_f}{\rho_m}\abs{ b(\eta^n,\rho^n,\phi_h^n) } =  \frac{2v_f}{3\rho_m} \abs{ \int_{\Omega} \eta^n \frac{\partial \rho^{n}}{\partial x}\phi_h^n{\rm d}x - \int_{\Omega} \eta^n  \frac{\partial \phi_h^n}{\partial x}\rho^{n}{\rm d}x }.
\eea
\noindent The first term in \eqref{eq: t3_1} is bounded as
\beas\label{eq: t3_2}
\abs{ \int_{\Omega} \eta^n \frac{\partial \rho^{n}}{\partial x}\phi_h^n{\rm d}x } &\leq& \int_{\Omega} \abs{ \eta^n \frac{\partial \rho^{n}}{\partial x}\phi_h^n}{\rm d}x \nonumber  \\
&\leq&  \norm{\eta^n} \norm{\phi_h^n} \norm{\frac{\partial \rho^{n}}{\partial x}}_{L^{\infty}} \leq C \norm{\eta^n} \norm{\phi_h^n},
\eeas
\noindent and the second term in \eqref{eq: t3_1} is bounded as 
\beas\label{eq: t3_3}
\abs{\int_{\Omega} \eta^n  \frac{\partial \phi_h^n}{\partial x}\rho^{n}{\rm d}x} &\leq&  \norm{\eta^n} \norm{\frac{\partial \phi_h^n}{\partial x}} \norm{\rho^n}_{L^{\infty}} \leq C \norm{\eta^n} \norm{\frac{\partial \phi_h^n}{\partial x}} \nonumber\\
&\leq& C \norm{\eta^n} \pare{\norm{ \pdv{x} \phi_h^{n^*} } + \norm{ \pdv{x} D_N^h \overline{\phi_h^n}^h}}.
\eeas
Therefore, by triangle inequality we have that $ |T_3|$ has the bound
\bea\label{eq:t3_bound}
|T_3| \leq C \frac{2v_f}{3\rho_m} \pare{\norm{\phi_h^n} + \norm{ \pdv{x} \phi_h^{n^*} } + \norm{ \pdv{x} D_N^h \overline{\phi_h^n}^h}}\norm{\eta^n}.
\eea
To bound $|T_4|$ we use formulas \eqref{eq: f4},\eqref{eq: f1} to get
\bea
|T_4|= \frac{2v_f}{\rho_m} |b(\phi_h^n,\rho^n, \phi_h^n) | = \frac{2v_f}{\rho_m} \abs{ \frac{2}{3} \int_{\Omega} \phi_h^n \frac{\partial \rho^{n}}{\partial x}\phi_h^n {\rm d} x+\frac{1}{3}\int_{\Omega} \rho^{n} \frac{\partial \phi_h^n}{\partial x}\phi_h^n {\rm d} x }\nonumber \\
\leq \frac{v_f}{\rho_m} \int_{\Omega} \abs{ \phi_h^n \frac{\partial \rho^{n}}{\partial x}\phi_h^n } {\rm d} x \leq \frac{v_f}{\rho_m} \norm{\frac{\partial \rho^{n}}{\partial x}}_{L^{\infty}} \norm{\phi_h^n}^2\leq \frac{v_f}{\rho_m}C  \norm{\phi_h^n}^2.
\eea
Using $\ph^n=I_h^n-\phi_h^n$ we obtain
\begin{equation}\label{eq4}
|T_5|= \frac{2v_f}{\rho_m} \abs{b(\ph^n,\eta^n,\phi_h^n)} = \frac{2v_f}{\rho_m}\abs{ b(I_h^n,\eta^n,\phi_h^n)-b(\phi_h^n,\eta^n,\phi_h^n) }.
\end{equation}
We estimate each term on the right hand side of \eqref{eq4} separately.
Starting with the first term, we can use formula \eqref{eq: f3} and bound similarly to $ |T_3|$
\bea\label{eq: t5b1}
\abs{ b(I_h^n,\eta^n,\phi_h^n) } &\leq& \frac{1}{3}\int_{\Omega} \abs{ \eta^n \frac{\partial I_h^n}{\partial x}\phi_h^n } {\rm d} x + \frac{2}{3}\int_{\Omega} \abs{ \eta^n \frac{\partial \phi_h^n }{\partial x}I_h^n} {\rm d} x \nonumber \\
&\leq&  \frac{1}{3} \norm{\eta^n} \norm{\phi_h^n} \norm{\pdv{x} I_h^n }_{L^\infty} + \frac{2}{3} \norm{ \eta^n } \norm{\pdv{x} \phi_h^n} \norm{I_h^n}_{L^\infty} \nonumber \\
&\leq&   C \pare{ \norm{\phi_h^n} + \norm{\pdv{x} \phi_h^{n^*} }+ \norm{\pdv{x} D_N^h \overline{\phi_h^n}^h } }  \norm{\eta^n}.
\eea
Following the same approach as in $|T_4|$,  the second term is bounded as
\bea\label{eq: t5b2}
\frac{2v_f}{\rho_m}\abs{ b(\phi_h^n, \eta^n, \phi_h^n) } \leq \frac{v_f}{\rho_m} \norm{\pdv{x} \eta^n }_{L^{\infty}} \norm{\phi_h^n}^2 \leq \frac{v_f}{\rho_m} C \norm{\phi_h^n}^2.
\eea
Combining \eqref{eq: t5b1} and \eqref{eq: t5b2} yields
\bea\label{t5_bound}
|T_5| \leq  C \frac{2v_f}{\rho_m} \pare{ \norm{\phi_h^n} + \norm{\pdv{x} \phi_h^{n^*} }+ \norm{\pdv{x} D_N^h \overline{\phi_h^n}^h } }  \norm{\eta^n} + \frac{v_f}{\rho_m} C \norm{\phi_h^n}^2.
\eea
To bound $ |T_6|$ we simply apply the Cauchy-Schwarz inequality followed by Young's inequality
\bea\label{t6_bound}
|T_6| \leq \chi\delta^2 \norm{\pdv{x} \eta^{n^*}} \norm{\pdv{x} \phi_h^{n^*}} \leq 3 \chi\delta^2 \norm{\pdv{x} \eta^{n^*}}^2 + \frac{\chi\delta^2}{12} \norm{\pdv{x} \phi_h^{n^*}}^{2}.
\eea
$ |T_7|$ is bounded with Cauchy-Schwarz and Young's inequalities, and an application of Taylor's remainder formula resulting in
\bea
 |T_7| & \leq & \frac{1}{4}\norm{\phi_h^{n}}^2 + \norm{\rho_t^{n}-\frac{1}{\Delta t}(\rho^{n}-\rho^{n-1})}^2 \nonumber \\
  & \leq & \frac{1}{4}\norm{\phi_h^{n}}^2 + C  \delt \int_{t^{n-1}}^{t^{n}} \norm{\rho_{tt}}^2 \, dt \label{t7_bound}.
\eea
Lastly, $|T_8|$ is bounded as
\bea
|T_8| &\leq& \chi\delta^2 \norm{\pdv{x} \rho^{n^*}} \norm{\pdv{x} \phi_h^{n^*}} \leq 3 \chi \delta^2 \norm{\pdv{x} \rho^{n^*}}^2 + \frac{\chi\delta^2}{12} \norm{\pdv{x} \phi_h^{n^*}}^{2}. \label{eq: t8_bound}
\eea
Applying the bounds \eqref{eq:t0_bound} - \eqref{eq: t8_bound} to \eqref{eq: error1} and simplifying results in
\bea\label{eq: error2}
  & &\frac{1}{2\delt} \pare{ \norm{\phi_h^n}^2 - \norm{\phi_h^{n-1}}^2 } + \frac{3\chi \delta^2}{4} \norm{ \pdv{x} \phi_h^{n^*}}^2 \leq 
  3\chi\delta^2 \pare{ \norm{\pdv{x} \rho^{n^*} }^2 + \norm{ \pdv{x} \eta^{n^*}}^2} \nonumber \\
   +&  &C\frac{v_f^2}{\delta^2} \pare{\frac{1}{\chi} +1 }\norm{\eta^n}^2 + \pare{ C\frac{v_f}{\rho_m} + \frac{3}{2} } \norm{\phi_h^n}^2 + C \delt \int_{t^{n-1}}^{t^n} \norm{ \rho_{tt} }^2 \; dt \nonumber \\
  +& &C \frac{v_f}{\rho_m} \pare{ \norm{\phi_h^n} + \norm{\pdv{x} \phi_h^{n^*} }+ \norm{\pdv{x} D_N^h \overline{\phi_h^n}^h } }  \norm{\eta^n} +{ \frac{1}{4\delt} \int_{t_{n-1}}^{t_n} \norm{\eta_t}^2{\rm d}t }. 
\eea
We further bound the last terms in \eqref{eq: error2} using Young's inequality and Lemma \ref{lem: Smoothing_lemma} giving
\bea\label{eq: lastterm}
&&C \frac{v_f}{\rho_m} \pare{ \norm{\phi_h^n} + \norm{\pdv{x} \phi_h^{n^*} }+ \norm{\pdv{x} D_N^h \overline{\phi_h^n}^h } }  \norm{\eta^n} \nonumber\\
&\leq&\frac{1}{4} \norm{\phi_h^n}^2  +  \frac{Cv_f^2}{\rho_m^2}  \norm{\eta^n }^2 + \frac{\chi \delta^2}{4} \norm{\pdv{x} \phi_h^{n^*}}^2 +\nonumber\\
&&\frac{Cv_f^2}{\rho_m^2 \chi \delta^2 } \norm{\eta^n }^2 + \frac{1}{4} \norm{\phi_h^n}^2 + \frac{Cv_f^2}{\rho_m^2 \delta^2} \norm{\eta^n}^2 .
\eea
Plugging in \eqref{eq: lastterm} into \eqref{eq: error2} yields
\bea\label{eq: error3}
  & &\frac{1}{2\delt} \pare{ \norm{\phi_h^n}^2 - \norm{\phi_h^{n-1}}^2 } + \frac{\chi \delta^2}{2} \norm{ \pdv{x} \phi_h^{n^*}}^2 \leq 
  3\chi\delta^2 \pare{ \norm{\pdv{x} \rho^{n^*} }^2 + \norm{ \pdv{x} \eta^{n^*}}^2} \nonumber \\
   +&  &C\frac{v_f^2}{\delta^2} \pare{\frac{1}{\chi} +1 }\norm{\eta^n}^2 + \pare{ C\frac{v_f}{\rho_m} + { 2} } \norm{\phi_h^n}^2 + C \delt \int_{t^{n-1}}^{t^n} \norm{ \rho_{tt} }^2 \; dt \nonumber \\
  +& &C \frac{v_f^2}{\rho_m^2} \pare{ 1+ \frac{1}{\chi\delta^2}{+\frac{1}{\delta^2}} }  \norm{\eta^n}^2+{ \frac{1}{4\delt} \int_{t_{n-1}}^{t_n} \norm{\eta_t}^2{\rm d}t }.
\eea
Multiplying by $2 \Delta t$, summing up from $n=1$ to $M$, and recalling that $\norm{\phi_h^0}=0$ gives
\bea
    & &\norm{\phi_h^M}^2 + \chi \delta^2 \delt \sum_{n=1}^{M} \norm{ \pdv{x} \phi_h^{n^*}}^2\nonumber\\
    &\leq& 6 \chi\delta^2 \delt \sum_{n=1}^{M} \pare{ \norm{\pdv{x} \rho^{n^*} }^2 + \norm{ \pdv{x} \eta^{n^*}}^2} +C\frac{v_f^2}{\delta^2} \pare{\frac{1}{\chi} +1 } \delt \sum_{n=1}^{M} \norm{\eta^n}^2\nonumber \\
    && + \pare{ C\frac{v_f}{\rho_m} + { 2}}\delt \sum_{n=1}^{M} \norm{\phi_h^n}^2 +C \delt^2 \int_{t^0}^{t^M} \norm{ \rho_{tt} }^2 \; dt\nonumber \\
    &&+C \frac{v_f^2}{\rho_m^2} \pare{ 1+ \frac{1}{\chi\delta^2}{ +\frac{1}{\delta^2}} } \delt \sum_{n=1}^{M} \norm{\eta^n}^2+{ \frac{1}{2} \int_{t^0}^{t^M} \norm{\eta_t}^2{\rm d}t }.  \qquad
\eea
Using the standard interpolation inequalities \eqref{eq: l2projbound}, \eqref{eq: h1projbound} and Lemma \ref{lem: deconvolution error}, we have 
\beas
    & &\norm{\phi_h^M}^2 + \chi \delta^2 \trinorm{ \pdv{x} \phi_h^{n^*}}_{2,0}^2 \leq \chi \delta^{4N+6} \trinorm{\pdv[\rho]{x}}^2_{2,2N+3}\nonumber\\
    &+&C(N)\chi(\delta^2 h^{2k} + h^{2k+2} ) \trinorm{ \rho }_{2,k+1}^2  \nonumber \\
  & + & 6 \chi\delta^2 h^{2k} \trinorm{ \rho}_{2,k+1}^2  + C\frac{v_f^2}{\delta^2} \pare{\frac{1}{\chi} +1 } h^{2k+2} \trinorm{\rho}_{2,k+1}^2\nonumber \\
    &+&C \pare{ \frac{v_f}{\rho_m} + 1 } \delt \sum_{n=1}^{M} \norm{\phi_h^n}^2 +C \delt^2 \int_{t^0}^{t^M} \norm{ \rho_{tt} }^2 \; dt \nonumber \\
    &+&
    C \frac{v_f^2}{\rho_m^2} \pare{ 1+ \frac{1}{\chi\delta^2}{+\frac{1}{\delta^2}} } h^{2k+2} \trinorm{\rho}_{2,k+1}^2    +C h^{2k+2} \int_{t^0}^{t^M} \norm{\rho_t} _{k+1}^2 {\rm d}t  . 
\eeas
Thus, with a sufficiently small time step ($ \gamma_n\delt :=  C(\frac{v_f}{\rho_m}+ 1)\delt\leq 1$), applying the Gronwall inequality from Lemma \ref{lem:discreteGronwall} results in
 \beas
    & &\hspace*{-.2cm}\norm{\phi_h^M}^2 + \chi \delta^2 \trinorm{ \pdv{x} \phi_h^{n^*}}_{2,0}^2\hspace*{-.2cm}\leq C\exp\pare{{\delt \sum_{n=0}^{M} \frac{\gamma}{1- \delt \gamma_n} }}\bigg( \chi \delta^{4N+6} \trinorm{\pdv[\rho]{x}}^2_{2,2N+3}\nonumber\\
    &&\hspace*{-.2cm} +\chi h^{2k+2}\trinorm{ \rho }_{2,k+1}^2+ \chi\delta^2 h^{2k} \trinorm{ \rho}_{2,k+1}^2\nonumber \\
    &&\hspace*{-.2cm}+ \frac{v_f^2}{\delta^2} \pare{\frac{1}{\chi} +1 } h^{2k+2} \trinorm{\rho}_{2,k+1}^2 +  h^{2k+2} \int_{t^0}^{t^M} \norm{\rho_t} _{k+1}^2 {\rm d}t \nonumber \\
    &&\hspace*{-.2cm}+ \delt^2 \int_{t^0}^{t^M} \norm{ \rho_{tt} }^2 \; dt 
    + \frac{v_f^2}{\rho_m^2} \pare{ 1+ \frac{1}{\chi\delta^2}{+\frac{1}{\delta^2}} } h^{2k+2} \trinorm{\rho}_{2,k+1}^2 \bigg). 
\eeas
Employing triangle inequality completes the proof.
\end{proof}

\begin{remark}\label{remark: convergence_rate}
For a smooth solution $\rho$ and dropping higher order terms, the above error estimate becomes
\begin{equation*}
 \trinorm{\rho - \rho_h}^2_{\infty,0} = O\pare{ \Delta t^2 + {\chi\delta^{4N+6}} + \chi\delta^2 h^{2k} + \chi^{-1}\delta^{-2} h^{2k+2} + \delta^{-2}h^{2k+2} }.
 \end{equation*}
Hence, to obtain the highest rates we choose $ \delta=h^{0.5}$, i.e.,

\begin{equation*}
 \trinorm{\rho - \rho_h}^2_{\infty,0} = O\pare{\delt^2 + \chi h^{2N+3} +(\chi+\chi^{-1}+1) h^{2k+1} } .
 \end{equation*}

\end{remark}

Next we will prove the convergence analysis for solutions of Algorithm \ref{alg: FDTFGM}.

\begin{theorem}\label{cor: Timefilter Lemma}
Let $ \ph^n $ be the solution of Algorithm \ref{alg: FDTFGM} given by \eqref{eq: FDTFFEM} and $\rho$ be a solution of \eqref{eq: problem_description} with enough regularity. Then, for sufficiently small $\delt $ we have

\beas
 &&\frac{1}{4} \pare{ \norm{e^M}^2  + \norm{2e^M - e^{M-1} }^2 + \norm{ e^M - e^{M-1} }^2 }  \nonumber \\
&+& \frac{3}{4}\sum_{n=2}^{M} \norm{ e^{n} - e^{n-1} - e^{n-2} }^2 \leq  CK \bigg(   h^{2k+2} \trinorm{\rho}_{\infty,k+1}^{2}  \nonumber \\
&+& \chi(\delta^2 h^{2k} + h^{2k+2} ) \trinorm{ \rho }_{2,k+1}^2  +\chi \delta^{4N+6} \trinorm{\pdv[\rho]{x}}^2_{2,2N+3}   \nonumber \\
&+&  \chi \delta^2 h^{2k} \trinorm{ \rho  }_{2,k+1}^2  + \frac{ v_f^2}{\delta^2} \pare{\frac{1}{\chi} +1 } h^{2k+2} \trinorm{\rho}_{2,k+1}^2  \nonumber \\
&+&  \frac{ v_f^2}{\rho_m^2} \pare{ 1+ \frac{1}{\chi\delta^2} + \frac{1}{\delta^2}} h^{2k+2} \trinorm{\rho}_{2,k+1}^2   + \delt^4 \int_{t^{0}}^{t^{M}} \norm{\rho_{ttt}}^2 \; dt \nonumber \\
&+&  \pare{v_f^2 + \frac{v_f}{\rho_m}} \delt^4 \int_{t^{0}}^{t^M} \norm{\pdv{x} \rho_{tt} }^2  dt  
+ \frac{ v_f}{\rho_m}  \delt ^4 \int_{t^{0}}^{t^M} \norm{\rho_{tt}}^2 \; dt \bigg).
\eeas
 where $ e^n = \rho^n - \rho_h^n$, and $K = \exp\pare{ \delt \sum_{n=0}^{M} \frac{ \gamma_n }{1-\delt \gamma_n } }$ and $ \gamma_n = C\pare{ \frac{v_f}{\rho_m} +1 } $.
\end{theorem}

\begin{proof}
Evaluating the weak formulation of \eqref{eq: problem_description} at time $ t^n$ yields
\bea
\frac{1}{\delt}\pare{D[\rho^n], \vh } + v_f \pare{ \pdv{x} I[\rho^{n}],v_h} - \frac{2v_f}{\rho_m} b(I[\rho^{n}],I[\rho^{n}],v_h) \nonumber \\
+ \chi \delta^2 \pare{ \pdv{x} I[\rho^{n}]^* , \pdv{x} v_h^*  } = \tau(\rho^n,\vh), \label{eq: true_sol}
\eea
where the consistency error, $\tau$, is given by
\beas
\tau(\rho^n,\vh) := \pare{ \frac{D[\rho^n]}{\delt} - \rho_t^n  ,\vh} + v_f \pare{ \pdv{x} \pare{I[\rho^{n}] - \rho^n },v_h} \nonumber \\
+ \frac{2v_f}{\rho_m} \pare{b(\rho^{n},\rho^{n},v_h) - b(I[\rho^{n}],I[\rho^{n}],v_h)} + \chi \delta^2 \pare{ \pdv{x} I[\rho^{n}]^* , \pdv{x} v_h^*  }.
\eeas
Subtracting \eqref{eq: FDTFFEM} from \eqref{eq: true_sol} gives the error equation,
\bea
\frac{1}{\delt}\pare{D[e^n], \vh } + v_f \pare{ \pdv{x} I[e^{n}],v_h} - \frac{2v_f}{\rho_m} b(I[e^{n}],I[\rho^{n}],v_h) \nonumber \\
- \frac{2v_f}{\rho_m} b(I[\rho_h^n],I[e^n],\vh) + \chi \delta^2 \pare{ \pdv{x} I[e^{n}]^* , \pdv{x} v_h^*  } = \tau(\rho^n,\vh) \label{eq: TF_error1}.
\eea
We again decompose the error as $ e^{n} = \rho^n - I_h^n+ I_h^n- \ph^n  = \eta^{n} + \phi^{n}_{h}$ and pick $ \vh = I[\phi_h^n]$ so that \eqref{eq: TF_error1} becomes
\bea
& &\frac{1}{\delt}\pare{D[\phi_h^n], I[\phi_h^{n}] } + \chi \delta^2 \norm{ \pdv{x} I[\phi_h^{n}]^* }^2 =  -\frac{1}{\delt}\pare{D[\eta^n],  I[\phi_h^{n}] } \nonumber \\
&-&  v_f \pare{ \pdv{x} I[\phi_h^{n}],  I[\phi_h^{n}] } - v_f \pare{ \pdv{x} I[\eta^{n}],  I[\phi_h^{n}] } \nonumber \\
&+& \frac{2v_f}{\rho_m} \pare{ b(I[\eta^{n}],I[\rho^{n}], I[\phi_h^{n}])  +  b(I[\phi_h^{n}],I[\rho^{n}], I[\phi_h^{n}]) + b(I[\rho_h^n],I[\eta^n], I[\phi_h^{n}]) } \nonumber \\
&-& \chi \delta^2 \pare{ \pdv{x} I[\eta^{n}]^* , \pdv{x}  I[\phi_h^{n}]^*  } + \tau(\rho^n, I[\phi_h^{n}]) \label{eq: TF_error2}
\eea

We now bound the first and last terms on the right-hand side, while the rest of them are handled similarly to the ones in Theorem \ref{thm:ErrorEstimate}.

Using Lemma \ref{lem: algebra_lemma} for the first term, we have
\beas \label{eq: timeterm}
\pare{D[\phi_h^n], I[\phi_h^{n}] } =  \E[\phi_h^n] - \E[\phi_h^{n-1}] + \Z[\phi_h^n].
\eeas

For the last term, the consistency error $ \tau(\rho^n, I[\phi_h^{n}])$, we have the following bounds
\beas
\pare{ \frac{D[\rho^n]}{\delt} - \rho_t^n  , I[\phi_h^{n}]} &\leq& \frac{1}{4\eps} \norm{ \frac{D[\rho^n]}{\delt} - \rho_t^n  }^2 + \eps \norm{I[\phi_h^{n}]}^2 \nonumber \\
&\leq&  \frac{1}{4\eps} \frac{6}{5} \delt^3 \int_{t^{n-2}}^{t^{n}} \norm{\rho_{ttt}}^2 \; dt + \eps \norm{I[\phi_h^{n}]}^2,
\eeas

\beas
v_f \pare{ \pdv{x} \pare{I[\rho^{n}] - \rho^n }, I[\phi_h^{n}] } &\leq& v_f^2 \frac{1}{4\eps} \norm{\pdv{x} \pare{I[\rho^{n}] - \rho^n} }^2 + \eps \norm{  I[\phi_h^{n}] }^2 \nonumber\\
& \leq&  \frac{v_f^2}{3\eps} \delt^3 \int_{t^{n-2}}^{t^n} \norm{\pdv{x} \rho_{tt} }^2  dt + \eps \norm{  I[\phi_h^{n}] }^2.  
\eeas
The trilinear terms are bounded similarly to $T_4$
\beas
& & b(\rho^n,\rho^n,I[\phi_h^n]) - b(I[\rho^n],I[\rho^n],I[\phi_h^n])  \nonumber\\
& = &  b(\rho^n - I[\rho^n  ],\rho^n,I[\phi_h^n]) - b(I[\rho^n],\rho^n - I[\rho^n],I[\phi_h^n]) \nonumber \\
& \leq & \norm{\rho^n - I[\rho^n]} \norm{\pdv{x} \rho^n}_{L^\infty} \norm{I[\phi_h^n] } + {\norm{I[\rho^n]}_{L^\infty}} \norm{ \pdv{x} \pare{\rho^n - I[\rho^n]}} \norm{I[\phi_h^n]} \nonumber \\
& \leq & \norm{\rho^n - I[\rho^n]} C \norm{I[\phi_h^n] } + C \norm{\pdv{x}\pare{\rho^n - I[\rho^n]}} \norm{I[\phi_h^n]} \nonumber \\
& \leq & \frac{C}{\eps} \norm{\rho^n - I[\rho^n]}^2  + \frac{C}{\eps}\norm{\pdv{x}\pare{\rho^n - I[\rho^n]}}^2  + \eps \norm{I[\phi_h^n]}^2 \nonumber \\
& \leq & \frac{C}{\eps} \delt ^3 \int_{t^{n-2}}^{t^n} \norm{\rho_{tt}}^2 \; dt  + \frac{C}{\eps} \delt^3 \int_{t^{n-2}}^{t^{n}} \norm{\pdv{x}\rho_{tt} }^2 \; dt + \eps \norm{I[\phi_h^n]}^2,
\eeas

\beas
\chi \delta^2 \pare{ \pdv{x} I[\rho^{n}]^* , \pdv{x} I[\phi_h^n]^*  } &\leq& \chi \delta^2 \norm{ \pdv{x} I[\rho^{n}]^*} \norm{ \pdv{x} I[\phi_h^n]^*  } \nonumber \\
& \leq& \frac{\chi \delta^2}{4 \eps} \norm{ \pdv{x} I[\rho^{n}]^*}^2 + \eps \chi \delta^2  \norm{ \pdv{x} I[\phi_h^n]^*  }^2.
\eeas
Combining all the previous bounds into \eqref{eq: TF_error2} and setting $\eps = \frac{1}{3}$ gives
\beas
\frac{1}{\delt}\pare{\E[\phi_h^n] - \E[\phi_h^{n-1}] + \Z[\phi_h^n]} + \frac{\chi \delta^2}{4} \norm{ \pdv{x} I[\phi_h^{n}]^* }^2 \leq \nonumber \\
C \chi \delta^2  \pare{\norm{\pdv{x} I[\rho^{n}]^* }^2 + \norm{ \pdv{x} I[\eta^{n}]^*  }^2 } + \frac{C v_f^2}{\delta^2} \pare{\frac{1}{\chi} +1 }\norm{I[\eta^n]}^2 \nonumber \\
+  \frac{C v_f^2}{\rho_m^2} \pare{ 1+ \frac{1}{\chi\delta^2} + \frac{1}{\delta^2}}  \norm{I[\eta^n]}^2 + C\pare{ \frac{ v_f}{\rho_m} + 1 } \norm{I[\phi_h^n]}^2 \nonumber \\
+ C \delt^3 \int_{t^{n-2}}^{t^{n}} \norm{\rho_{ttt}}^2 \; dt + C \pare{v_f^2 + \frac{v_f}{\rho_m}} \delt^3 \int_{t^{n-2}}^{t^n} \norm{\pdv{x} \rho_{tt} }^2  dt  \nonumber \\
+ \frac{C v_f}{\rho_m}  \delt ^3 \int_{t^{n-2}}^{t^n} \norm{\rho_{tt}}^2 \; dt \label{eq: TF_error3}.
\eeas

Multiplying by $ 4 \delt$ and summing from $n=2,\dots M$ yields 

\beas
4(\E[\phi^M] - \E[\phi^{1}]) + 4\delt \sum_{n=2}^{M} \Z[\phi^n] + \chi \delta^2 \delt\sum_{n=2}^M \norm{ \pdv{x} I[\phi_h^{n}]^* }^2 \leq \nonumber \\
C \chi \delta^2  \delt \sum_{n=2}^{M} \pare{\norm{\pdv{x} I[\rho^{n}]^* }^2 + \norm{ \pdv{x} I[\eta^{n}]^*  }^2 } + \frac{C v_f^2}{\delta^2} \pare{\frac{1}{\chi} +1 } \delt \sum_{n=2}^{M} \norm{I[\eta^n]}^2 \nonumber \\
+  \frac{C v_f^2}{\rho_m^2} \pare{ 1+ \frac{1}{\chi\delta^2} + \frac{1}{\delta^2}} \delt \sum_{n=2}^{M} \norm{I[\eta^n]}^2 \nonumber \\+ C\pare{ \frac{ v_f}{\rho_m} + 1 } \delt \sum_{n=2}^{M} \norm{I[\phi_h^n]}^2 
 + C \delt^4 \int_{t^{0}}^{t^{M}} \norm{\rho_{ttt}}^2 \; dt \nonumber \\
+ C \pare{v_f^2 + \frac{v_f}{\rho_m}} \delt^4 \int_{t^{0}}^{t^M} \norm{\pdv{x} \rho_{tt} }^2  dt  
+ \frac{C v_f}{\rho_m}  \delt ^4 \int_{t^{0}}^{t^M} \norm{\rho_{tt}}^2 \; dt \label{eq: TF_error4}.
\eeas
Applying the standard interpolation inequalities \eqref{eq: l2projbound}, \eqref{eq: h1projbound}, and Lemma \ref{lem: deconvolution error}, we have
\beas
4(\E[\phi^M] - \E[\phi^{1}]) + 4\delt \sum_{n=2}^{M} \Z[\phi^n] + \chi \delta^2 \delt\sum_{n=2}^M \norm{ \pdv{x} I[\phi_h^{n}]^* }^2 \leq \nonumber \\
 + C(N)\chi(\delta^2 h^{2k} + h^{2k+2} ) \trinorm{ \rho }_{2,k+1}^2  +\chi \delta^{4N+6} \trinorm{\pdv[\rho]{x}}^2_{2,2N+3} + C \chi \delta^2 h^{2k} \trinorm{ \rho  }_{2,k+1}^2   \nonumber \\
+ \frac{C v_f^2}{\delta^2} \pare{\frac{1}{\chi} +1 } h^{2k+2} \trinorm{\rho}_{2,k+1}^2 +  \frac{C v_f^2}{\rho_m^2} \pare{ 1+ \frac{1}{\chi\delta^2} + \frac{1}{\delta^2}} h^{2k+2} \trinorm{\rho}_{2,k+1}^2  \nonumber \\
+ C\pare{ \frac{ v_f}{\rho_m} + { 1} } \delt \sum_{n=2}^{M} \norm{I[\phi_h^n]}^2 +C \delt^4 \int_{t^{0}}^{t^{M}} \norm{\rho_{ttt}}^2 \; dt \nonumber \\
+ C \pare{v_f^2 + \frac{v_f}{\rho_m}} \delt^4 \int_{t^{0}}^{t^M} \norm{\pdv{x} \rho_{tt} }^2  dt  
+ \frac{C v_f}{\rho_m}  \delt ^4 \int_{t^{0}}^{t^M} \norm{\rho_{tt}}^2 \; dt \label{eq: TF_error5}.
\eeas

Thus, applying the Gronwall inequality from Lemma \ref{lem:discreteGronwall} with a sufficiently small time step ($ \gamma_n\delt :=  C(\frac{v_f}{\rho_m}+ 1)\delt\leq 1$), we obtain
\beas
4\E[\phi^M]  + 4\delt \sum_{n=2}^{M} \Z[\phi^n] + \chi \delta^2 \delt\sum_{n=2}^M \norm{ \pdv{x} I[\phi_h^{n}]^* }^2 \leq \nonumber \\
 + CK \bigg( \E[\phi^1]  + \chi(\delta^2 h^{2k} + h^{2k+2} ) \trinorm{ \rho }_{2,k+1}^2  +\chi \delta^{4N+6} \trinorm{\pdv[\rho]{x}}^2_{2,2N+3}   \nonumber \\
+  \chi \delta^2 h^{2k} \trinorm{ \rho  }_{2,k+1}^2  + \frac{ v_f^2}{\delta^2} \pare{\frac{1}{\chi} +1 } h^{2k+2} \trinorm{\rho}_{2,k+1}^2  \nonumber \\
+  \frac{ v_f^2}{\rho_m^2} \pare{ 1+ \frac{1}{\chi\delta^2} + \frac{1}{\delta^2}} h^{2k+2} \trinorm{\rho}_{2,k+1}^2   + \delt^4 \int_{t^{0}}^{t^{M}} \norm{\rho_{ttt}}^2 \; dt \nonumber \\
+  \pare{v_f^2 + \frac{v_f}{\rho_m}} \delt^4 \int_{t^{0}}^{t^M} \norm{\pdv{x} \rho_{tt} }^2  dt  
+ \frac{ v_f}{\rho_m}  \delt ^4 \int_{t^{0}}^{t^M} \norm{\rho_{tt}}^2 \; dt \bigg) \label{eq: TF_error6},
\eeas
where $K = \exp\pare{ \delt \sum_{n=0}^{M} \frac{ \gamma_n }{1-\delt \gamma_n } }$. Note that $\E[\phi^1]$ is bounded as in \cite{decaria2020time}, i.e.,
\beas
\E[\phi^1] \leq C \pare{ \norm{ \rho^1 - \rho_h^1}^2 + \norm{\rho^0 - \rho_h^0}^2 }+ C h^{2k+2} \trinorm{\rho}_{\infty,k+1}^{2},
\eeas
where the error in the first time step is bounded by employing Theorem \ref{thm:ErrorEstimate} with $M=1$.
Applying triangle inequality and substituting \ref{eq: epsDef} and \ref{eq: zetadef} completes the proof.
\end{proof}

\begin{remark}\label{remark: convergence_rate_2}
For a smooth solution $\rho$, after dropping higher order terms and again picking $\delta=h^{0.5}$, the error estimate transforms into
\begin{equation*}
 \trinorm{\rho - \rho_h}^2_{\infty,0} = O\pare{\Delta t^4 + \chi h^{2N+3} +(\chi+\chi^{-1}+1) h^{2k+1}}.
 \end{equation*}
\end{remark}

\section{Computational Experiments}
\label{sec: numerical}
    \subsection{Convergence Rate Simulations}
        We now turn our attention to verifying the results of Theorem \ref{thm:ErrorEstimate} on the 1D domain $[0,1]$. We employ the manufactured solution $\rho(x,t)=\sin^4(\pi x)\sin(t)$ to the Dirichlet problem
            \[
                \begin{cases}
                    \frac{\partial\rho}{\partial t}+\left(v_f-\frac{2v_f}{\rho_m}\rho\right)\frac{\partial\rho}{\partial x}=f,&x\in(0,1),t>0,\\
                    \rho(0,t)=\rho(1,t)=0,&x\in\{0,1\},t>0,\\
                    \rho(x,0)=0,&x\in[0,1],t=0,
                \end{cases}
            \]
        where $v_f=\rho_m=1$, and
            \begin{align*}
                f(x,t)=&\sin^4(\pi x)\cos(t)+4\pi\cos(\pi x)\sin^3(\pi x)\sin(t)\left(v_f-\frac{2v_f}{\rho_m}\sin^4(\pi x)\sin(t)\right).
            \end{align*}
        Based on Remark \ref{remark: convergence_rate}, we observe the spatial rate of convergence, in the $P2$ finite element space with order of deconvolution $N=1$, to be $\mathcal{O}(h^{2.5})$ with stabilization and $\mathcal{O}(h^{2})$ without stabilization. Tables \ref{table:space_rates_chi=0}-\ref{table:space_rates_chi=1} exhibit these rates for a simulation where final time $T=0.02,\,\delt=5\times10^{-6}$, and $\delta=0.1\sqrt{h}$.
            \begin{table}[ht]
                \renewcommand{\arraystretch}{1.3}
                \hspace*{-0.7cm}
                \parbox{.45\linewidth}{
                    \tabcolsep=.1cm
                    \tiny
                    \begin{tabular}{ccccc}
                        \toprule
                        $\mathbf{h}$ & \multicolumn{2}{c}{\bfseries No time filter ($\mathbf{\gamma=0}$)} & \multicolumn{2}{c}{\bfseries Time filter ($\mathbf{\gamma=\frac{2}{3}}$)} \\
                        \cmidrule(r){2-3}\cmidrule(l){4-5}
                                        & $\norm{\rho-\rho_h}_{\ell^\infty(0,T;L^2)}$ & Rate & $\norm{\rho-\rho_h}_{\ell^\infty(0,T;L^2)}$  &  Rate\\
                        \midrule
                               $\frac{1}{6}$   & 9.58$\times10^{-5}$ & -- & 9.28$\times10^{-5}$ &  -- \\
                               $\frac{1}{12}$  & 1.46$\times10^{-5}$ & 2.71 & 1.34$\times10^{-5}$ & 2.79    \\
                               $\frac{1}{24}$  & 2.67$\times10^{-6}$ & 2.45 & 2.68$\times10^{-6}$ & 2.32 \\
                               $\frac{1}{48}$  & 6.23$\times10^{-7}$ & 2.10 & 6.26$\times10^{-7}$ & 2.10  \\
                               $\frac{1}{96}$  & 1.55$\times10^{-7}$ & 2.01 & 1.53$\times10^{-7}$ & 2.03  \\
                               $\frac{1}{192}$ & 3.84$\times10^{-8}$ & 2.01 & 3.82$\times10^{-8}$ & 2.00 \\
                        \bottomrule
                    \end{tabular}\caption{Space rates for $\chi=0$}\label{table:space_rates_chi=0}
                }
                \hspace*{1.1cm}
                \parbox{.45\linewidth}{
                    \tabcolsep=.1cm
                    \tiny
                    \begin{tabular}{ccccc}
                        \toprule
                        $\mathbf{h}$ & \multicolumn{2}{c}{\bfseries No time filter ($\mathbf{\gamma=0}$)} & \multicolumn{2}{c}{\bfseries Time filter ($\mathbf{\gamma=\frac{2}{3}}$)} \\
                        \cmidrule(r){2-3}\cmidrule(l){4-5}
                                        & $\norm{\rho-\rho_h}_{\ell^\infty(0,T;L^2)}$ & Rate & $\norm{\rho-\rho_h}_{\ell^\infty(0,T;L^2)}$  & Rate\\
                        \midrule
                               $\frac{1}{6}$   & 8.78$\times10^{-5}$ &  --      & 8.42$\times10^{-5}$ &  -- \\
                               $\frac{1}{12}$  & 1.35$\times10^{-5}$ & 2.70 & 1.49$\times10^{-5}$ & 2.50 \\
                               $\frac{1}{24}$  & 2.54$\times10^{-6}$ & 2.41 & 2.65$\times10^{-6}$ & 2.49 \\
                               $\frac{1}{48}$  & 5.51$\times10^{-7}$ & 2.20 & 5.46$\times10^{-7}$ & 2.28 \\
                               $\frac{1}{96}$  & 1.12$\times10^{-7}$ & 2.29 & 1.12$\times10^{-7}$ & 2.28 \\
                               $\frac{1}{192}$ & 2.10$\times10^{-8}$ & 2.42 & 2.10$\times10^{-8}$ & 2.42 \\
                        \bottomrule
                    \end{tabular}
                    \caption{Space rates for $\chi=1$}
                    \label{table:space_rates_chi=1}
                }
            \end{table}
        
        For the time convergence rate simulations, the finite element space is $P2$ and we set order of deconvolution $N=1$, mesh size $h=\frac{1}{100}$, filter width $\delta=0.1\sqrt{h}$, and final time $T=1$.
        
        Tables \ref{table:time_rates_chi=0}-\ref{table:time_rates_chi=1} display the expected time convergence rates without and with stabilization (i.e. $\chi=0$ and $\chi=1$ respectively), where the rate is 1 without time filtering, and 2 with time filtering. We observe that the stabilization makes little difference in both the errors and the rates.

        \begin{table}[ht]
            \renewcommand{\arraystretch}{1.3}
            \hspace*{-0.7cm}
            \parbox{.45\linewidth}{
                \tabcolsep=.1cm
                \tiny
                \begin{tabular}{ccccc}
                    \toprule
                    $\mathbf{\Delta t}$ & \multicolumn{2}{c}{\bfseries No time filter ($\mathbf{\gamma=0}$)} & \multicolumn{2}{c}{\bfseries Time filter ($\mathbf{\gamma=\frac{2}{3}}$)} \\
                    \cmidrule(r){2-3}\cmidrule(l){4-5}
                                    & $\norm{\rho-\rho_h}_{\ell^\infty(0,T;L^2)}$ & Rate & $\norm{\rho-\rho_h}_{\ell^\infty(0,T;L^2)}$  &  Rate\\
                    \midrule
                           $\frac{1}{10}$  & $1.97\times10^{-2}$ & --      & $4.88\times10^{-3}$ & --\\
                           $\frac{1}{20}$  & $9.13\times10^{-3}$ & 1.11 & $1.26\times10^{-3}$ & 1.95  \\
                           $\frac{1}{40}$  & $4.43\times10^{-3}$ & 1.04 & $3.29\times10^{-4}$ & 1.94  \\
                           $\frac{1}{80}$  & $2.19\times10^{-3}$ & 1.02 & $8.48\times10^{-5}$ & 1.96  \\
                           $\frac{1}{160}$  & $1.09\times10^{-3}$ & 1.01 & $2.31\times10^{-5}$ & 1.88  \\
                    \bottomrule
                \end{tabular}\caption{Time rates for $\chi=0$}\label{table:time_rates_chi=0}
            }
            \hspace*{1.1cm}
            \parbox{.45\linewidth}{
                \tabcolsep=.1cm
                \tiny
                \begin{tabular}{ccccc}
                    \toprule
                    $\mathbf{\Delta t}$ & \multicolumn{2}{c}{\bfseries No time filter ($\mathbf{\gamma=0}$)} & \multicolumn{2}{c}{\bfseries Time filter ($\mathbf{\gamma=\frac{2}{3}}$)} \\
                    \cmidrule(r){2-3}\cmidrule(l){4-5}
                                    & $\norm{\rho-\rho_h}_{\ell^\infty(0,T;L^2)}$ & Rate & $\norm{\rho-\rho_h}_{\ell^\infty(0,T;L^2)}$  & Rate\\
                    \midrule
                           $\frac{1}{10}$  & $1.96\times10^{-2}$  & --      & $4.87\times10^{-3}$ & --\\
                           $\frac{1}{20}$  & $9.12\times10^{-3}$  & 1.10 & $1.26\times10^{-3}$ & 1.95  \\
                           $\frac{1}{40}$  & $4.43\times10^{-3}$ & 1.04 & $3.29\times10^{-4}$ & 1.94  \\
                           $\frac{1}{80}$  & $2.19\times10^{-3}$ & 1.02 & $8.43\times10^{-5}$  & 1.96  \\
                           $\frac{1}{160}$  & $1.09\times10^{-3}$ & 1.01 & $2.15\times10^{-5}$  & 1.97  \\
                    \bottomrule
                \end{tabular}
                \caption{Time rates for $\chi=1$}
                \label{table:time_rates_chi=1}
            }
        \end{table}

    \subsection{Rarefaction}
    
        Here we examine the LWR model with the Greenshield's flow velocity relationship, with parameters $\rho_m=1,v_f=1,$ on the domain [0,1]. The initial and boundary conditions are $\rho(x, 0) = 0,$ for all $x \in [0,1]$ and  $\rho(0, t) = 0.47$ for $t>0$.  The solution obtained from this configuration is a rarefaction wave, expressed as
            \begin{align*}
                \rho(x, t) = 
                \begin{cases} 
                    0.47, & \text{for } x \leq 0.06t, \\
                    \frac{1}{2} - \frac{x}{2t}, & \text{for } 0.06t < x < t, \\
                    0, & \text{for } x \geq t.
                \end{cases} 
            \end{align*}
        
        In the context of the motivating biological application, this example models the case where the DNA strand is initially empty since the initial density is zero.  The boundary condition indicates that the initiation process begins immediately with RNAPs initiating on the strand at a constant rate so that the density at $x=0$ is fixed at $\rho(0, t) = 0.47$.   
        
        We use $P1$ elements for our Algorithm \ref{alg: FDTFGM} without time filtering. We also set $h=\frac{1}{128},\Delta t=10^{-4},$ and $\delta=\sqrt{h}$. Figures \ref{fig:RarefactionHalfTime} and \ref{fig:RarefactionFullTime} show simulations at times $t=0.5$ and $t=1$ based on different values of the stabilization parameter $\chi$.
        
            \begin{figure}[h!]
                \begin{subfigure}{.5\textwidth}
                    \centering
                    \includegraphics[clip,scale=0.08]{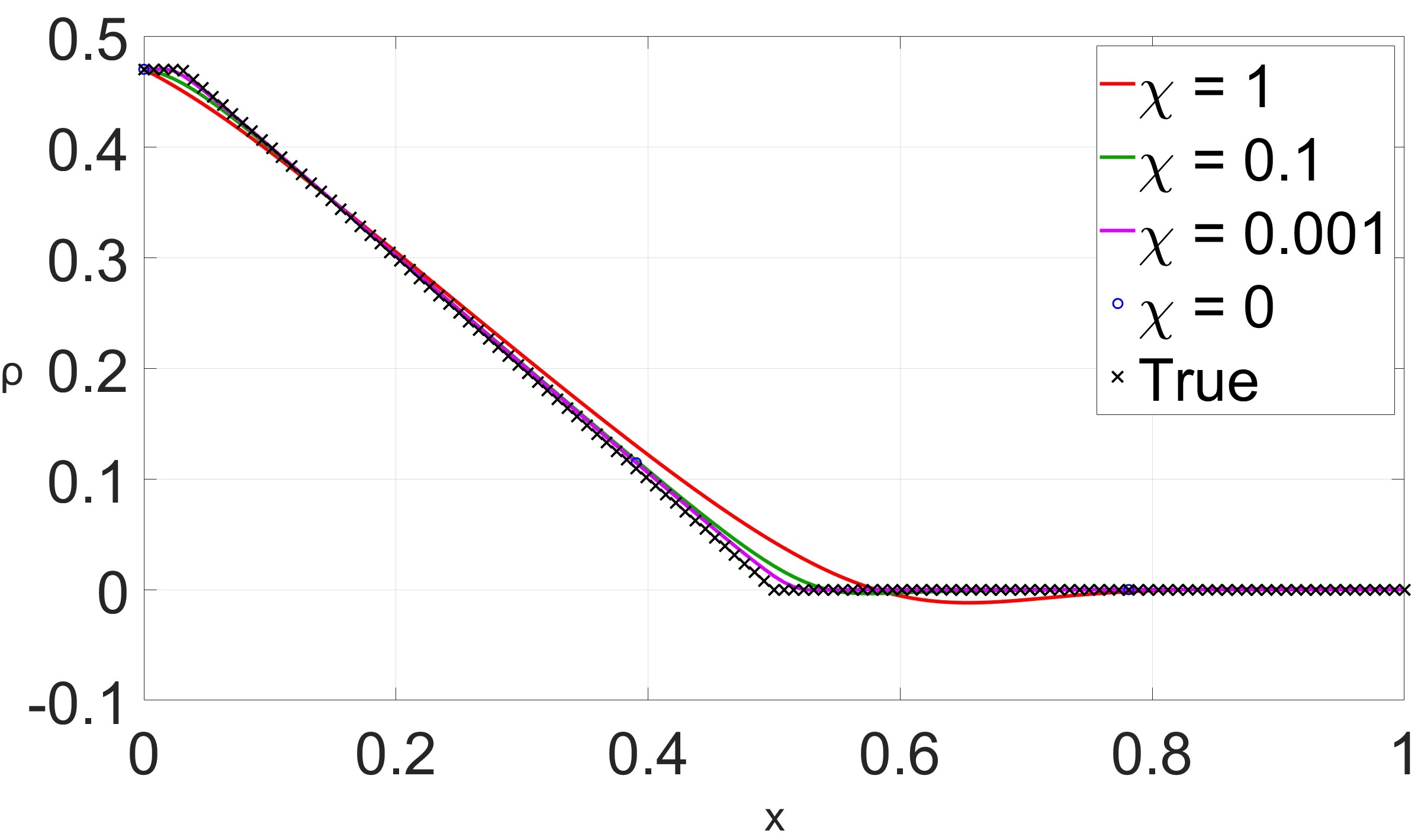}
                    \caption{$t=0.5$ with $N=0$}
                    \label{fig:RarefactionHalfTime}
                \end{subfigure}
                \begin{subfigure}{.5\textwidth}
                    \centering
                    \includegraphics[clip,scale=0.08]{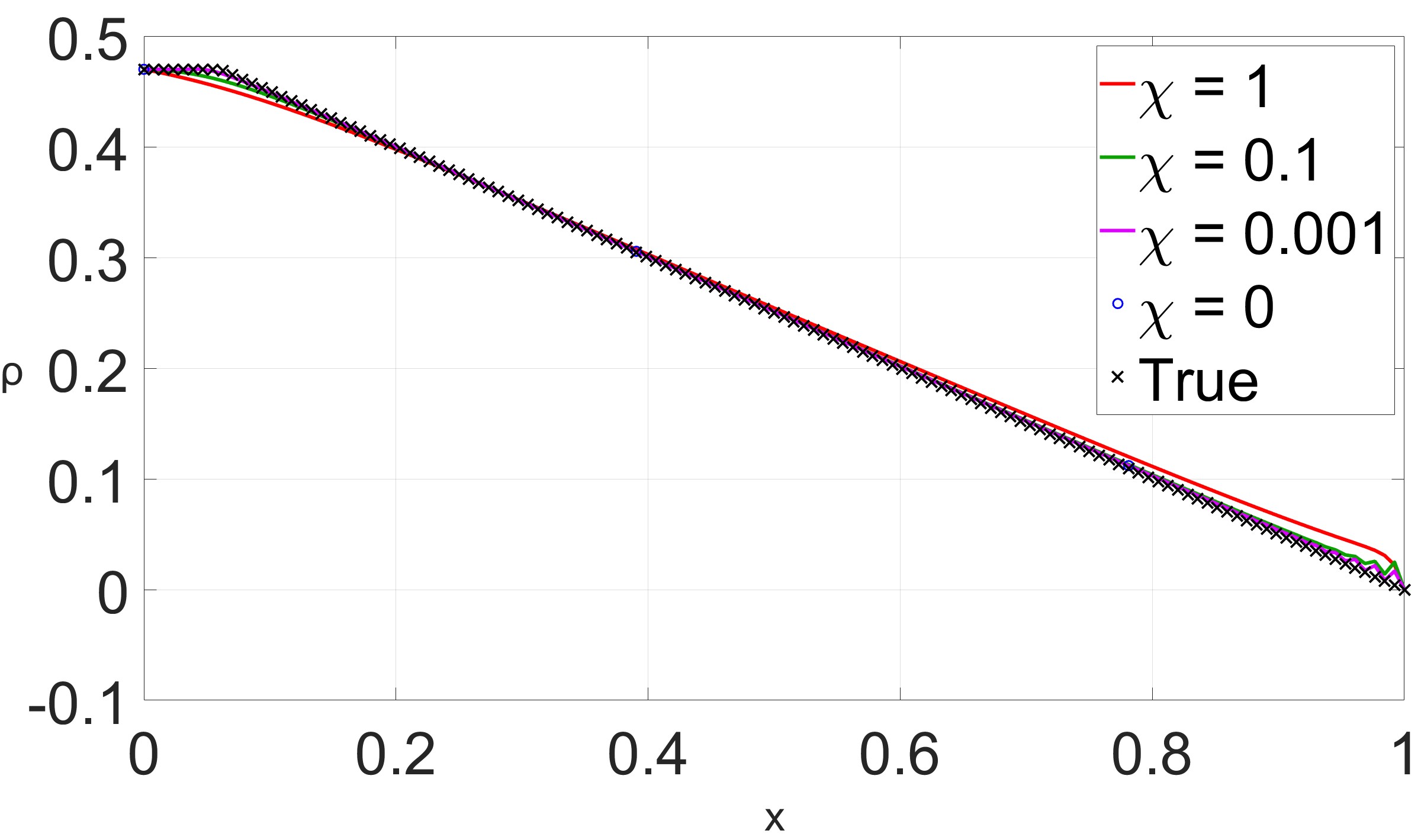}
                    \caption{$t=1$ with $N=0$}
                    \label{fig:RarefactionFullTime}
                \end{subfigure}
                \caption{Rarefaction density profile for various $\chi$}
            \end{figure}
        As can be somewhat seen in Figures \ref{fig:RarefactionHalfTime} and \ref{fig:RarefactionFullTime}, the finite element density profile begins to diverge from the true density profile as $\chi$ increases.

        Figures \ref{fig:L2ErrorsRarefactionN0P1}-\ref{fig:L2ErrorsRarefactionN1P2} examine the $L^2$ error further by plotting $\norm{\rho_h(\cdot,t)-\rho(\cdot,t)}$ over time for order of deconvolution $N=0$ and $N=1$, combined with $P1$ and $P2$ elements respectively. We observe that the errors for $N=1$ are smaller than the errors for $N=0$. Moreover, very small values of $\chi$ for stabilization or \textit{no} stabilization yield the best accuracy for the rarefaction solution. This is to be expected, as the rarefaction solution is continuous. In the next section, we will see an example of a discontinuous solution that enjoys the effects of stabilization. 
    
        \begin{figure}[h!]
            \begin{subfigure}{.5\textwidth}
                \hspace*{-.5cm}
                \includegraphics[clip,scale=0.08]{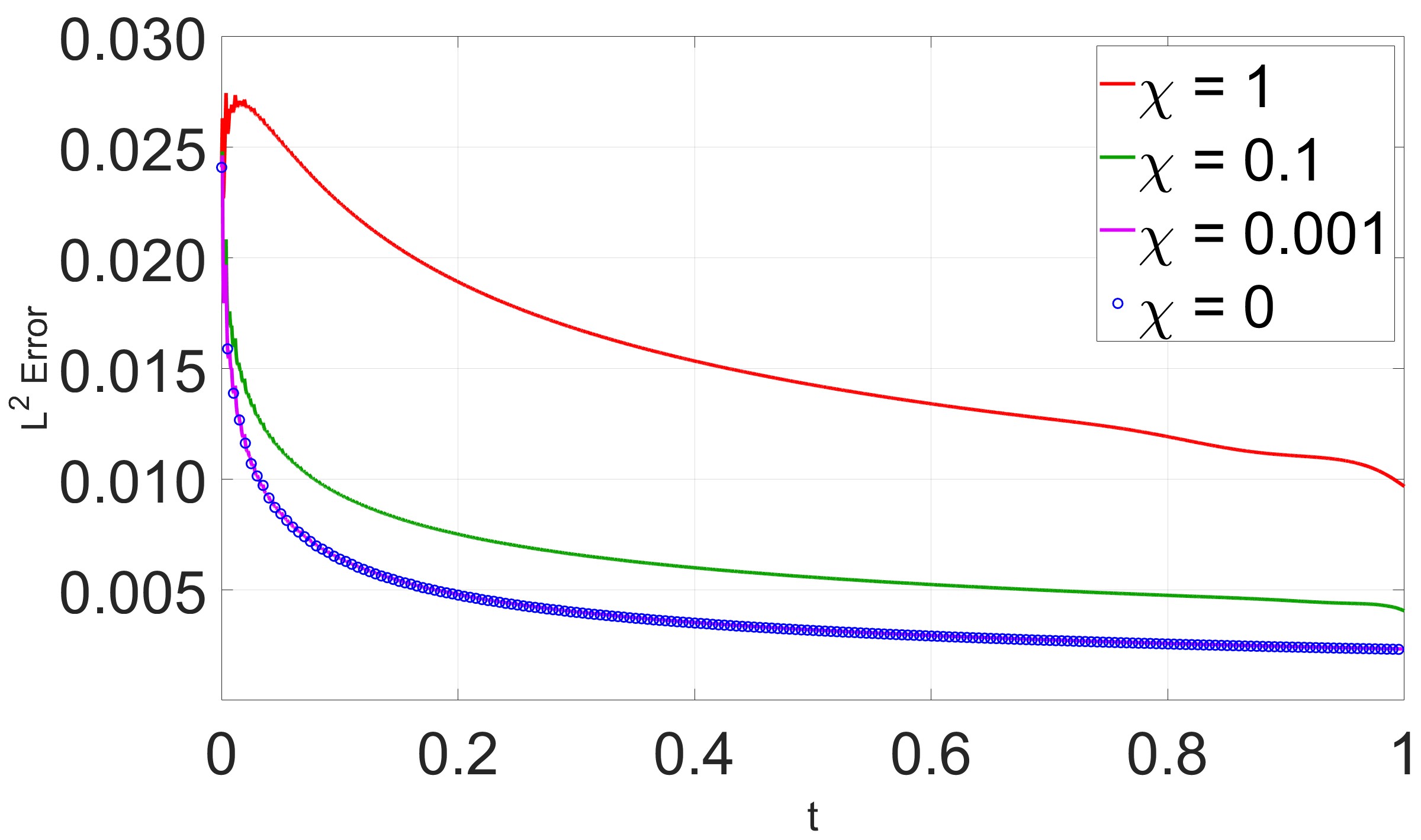}
                \caption{$N=0$ and $P1$ elements}
                \label{fig:L2ErrorsRarefactionN0P1}
            \end{subfigure}
            \begin{subfigure}{.5\textwidth}
                \hspace*{-.5cm}
                \includegraphics[clip,scale=0.08]{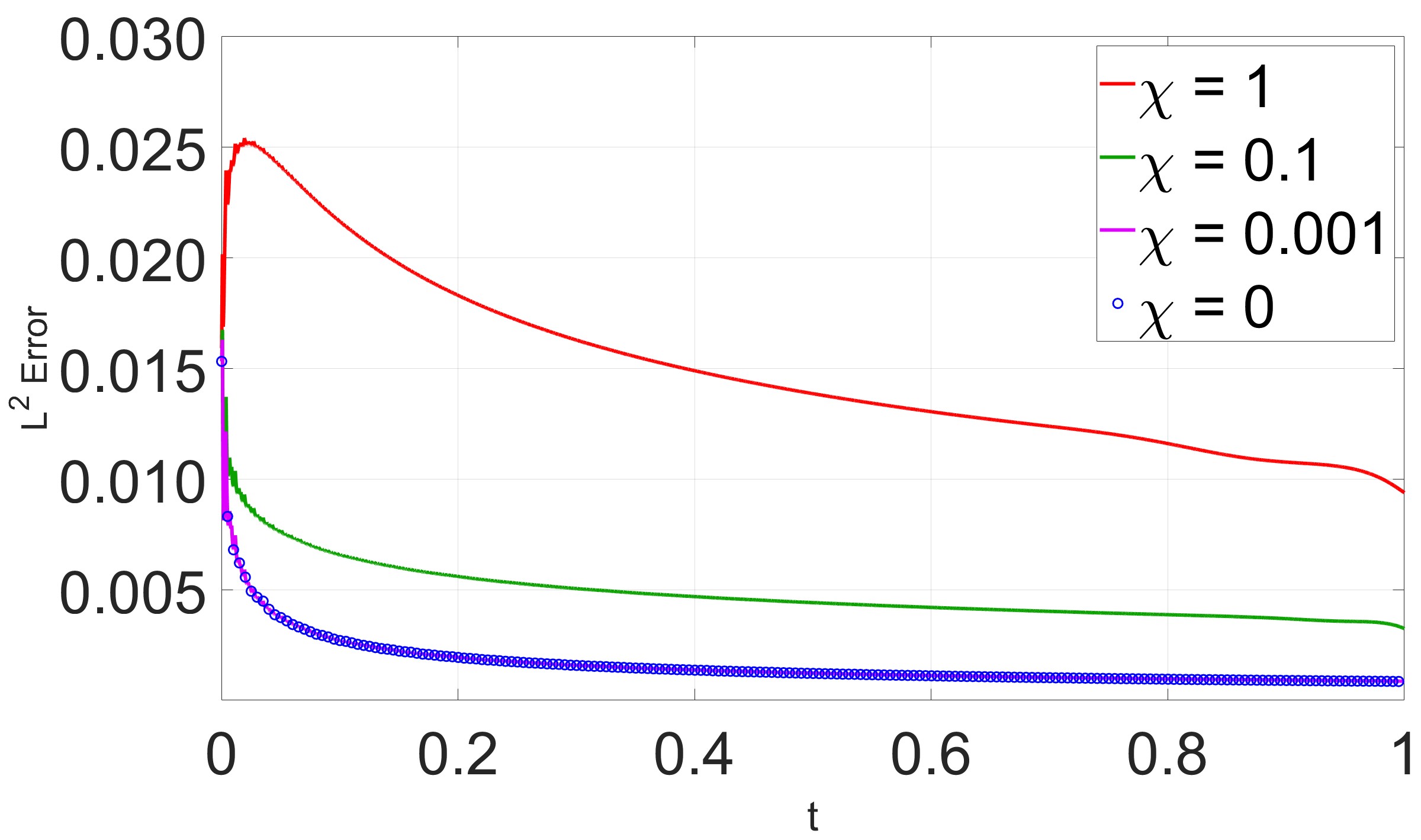}
                \caption{$N=0$ and $P2$ elements}
                \label{fig:L2ErrorsRarefactionN0P2}
            \end{subfigure}
            \begin{subfigure}{.5\textwidth}
                \hspace*{-.5cm}
                \includegraphics[clip,scale=0.08]{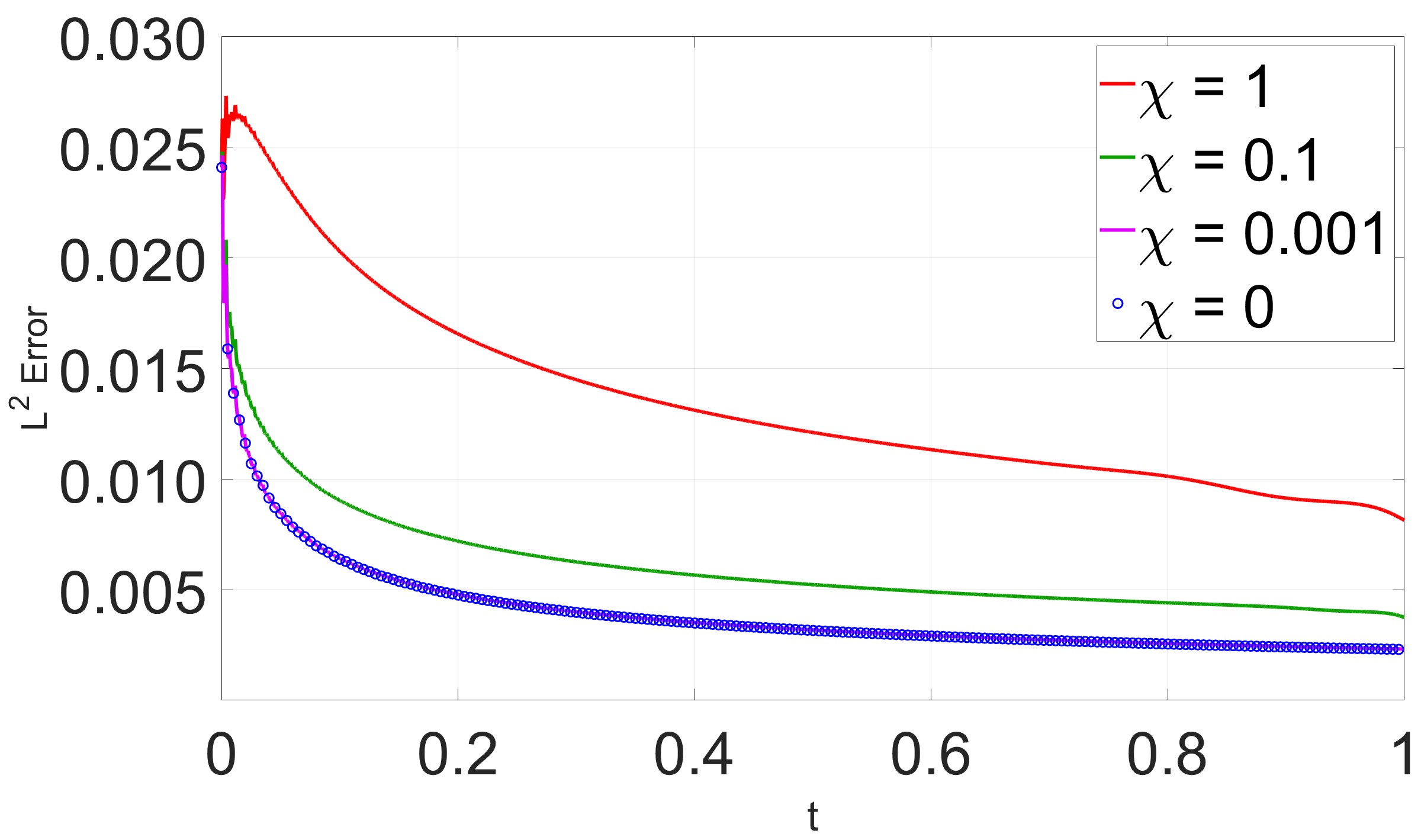}
                \caption{$N=1$ and $P1$ elements}
                \label{fig:L2ErrorsRarefactionN1P1}
            \end{subfigure}
            \begin{subfigure}{.5\textwidth}
                \hspace*{-.5cm}
                \includegraphics[clip,scale=0.08]{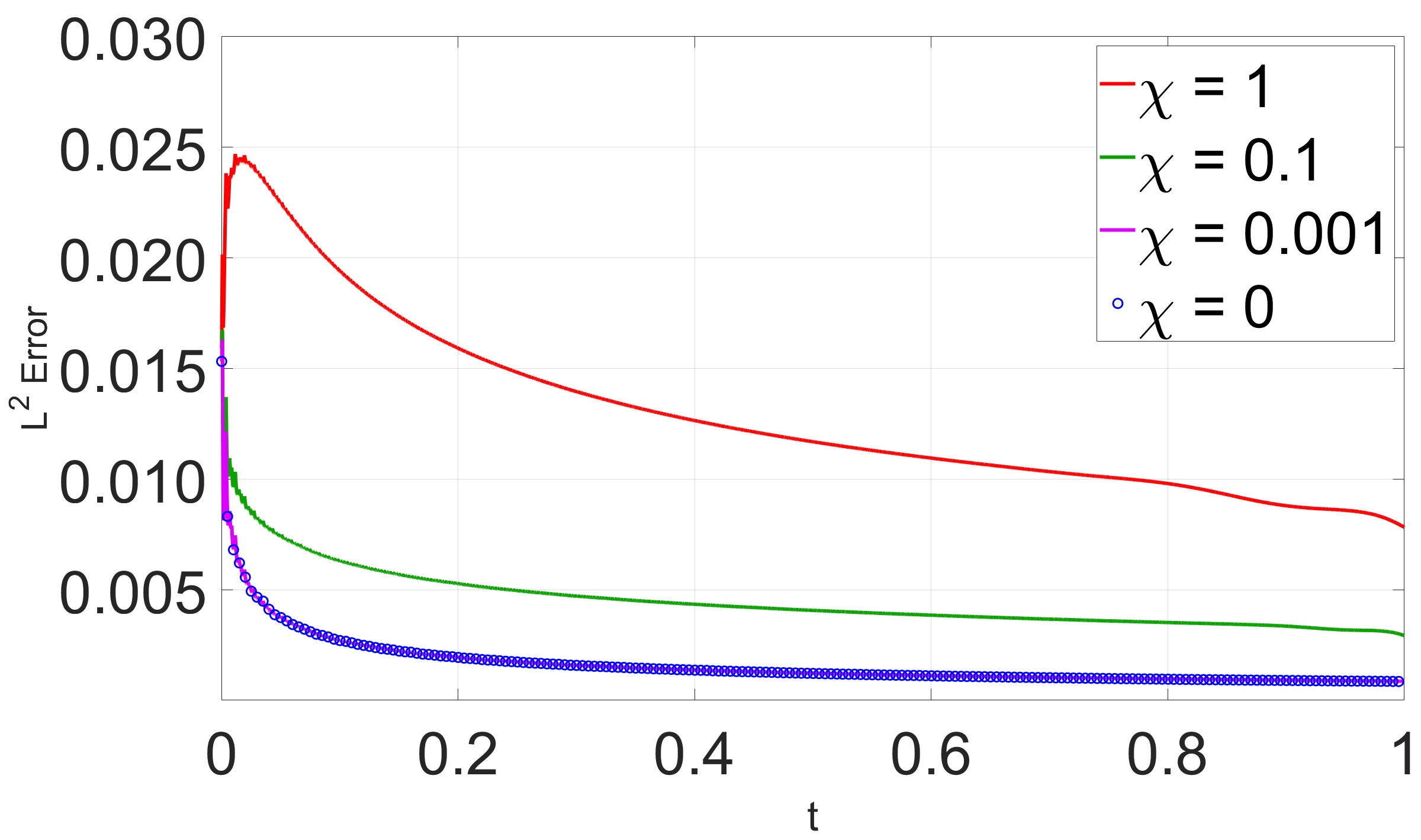}
                \caption{$N=1$ and $P2$ elements}
                \label{fig:L2ErrorsRarefactionN1P2}
            \end{subfigure}
            \caption{Plots of error in $\norm{\cdot}_{2}$ over time}
        \end{figure}

    \subsection{Shock Wave}
        Herein, we explore a shock wave solution to 
        \eqref{eq: problem_description} on the domain $[0,1]$. The parameters $\rho_m=1$ and $v_f=1$. We impose boundary and initial conditions $\rho(x, 0) = \frac{1}{3} \text{ on } (0,1], \; \rho(0, t) = \frac{1}{4}$. The shock wave is
            \begin{align*}
                \rho(x, t) = 
                \begin{cases} 
                    \frac{1}{4} & \text{if } x \leq \frac{5}{12}t, \\
                    \frac{1}{3} & \text{if } x > \frac{5}{12}t.
                \end{cases} 
            \end{align*}

        In the context of the biological application, this example corresponds to the case where the strand is initially occupied by RNAPs so that, on average, one third of it is being transcribed.  As time moves forward, the initiation rate is decreased so that the density at the initiation site is maintained at $\rho(0,t)= \frac{1}{4}$.  The result is a travelling shock wave with a discontinuity that propagates across the domain with shock speed of $\frac{5}{12}$.
        
        We use Algorithm \ref{alg: FDTFGM} without time filtering, with order of deconvolution $N=0$ and $N=1$, combined with $P1$ and $P2$ elements. Furthermore, $h=\frac{1}{128},\Delta t=10^{-4},$ and $\delta=\sqrt{h}$. Figures \ref{fig:ShockWaveHalfTimeN0P1}-\ref{fig:ShockWaveFullTimeN1P2} showcase how adjusting the stabilization parameter affects the finite element solution for the shock wave problem. The advantage of Vreman stabilization is evident in each of these figures, as increasing $\chi$ transforms the finite element solution from a highly oscillatory graph to a smooth shock wave. Additionally, we observe that for a fixed degree of finite elements, increasing the order of deconvolution $N$ does not visibly improve the oscillations. On the other hand, for both $N=1$ and $N=2$, increasing the degree of polynomials from $P1$ to $P2$ finite elements reduces the oscillations.
            \begin{figure}[!htb]
                \begin{subfigure}{.5\textwidth}
                    \includegraphics[clip,scale=0.08]{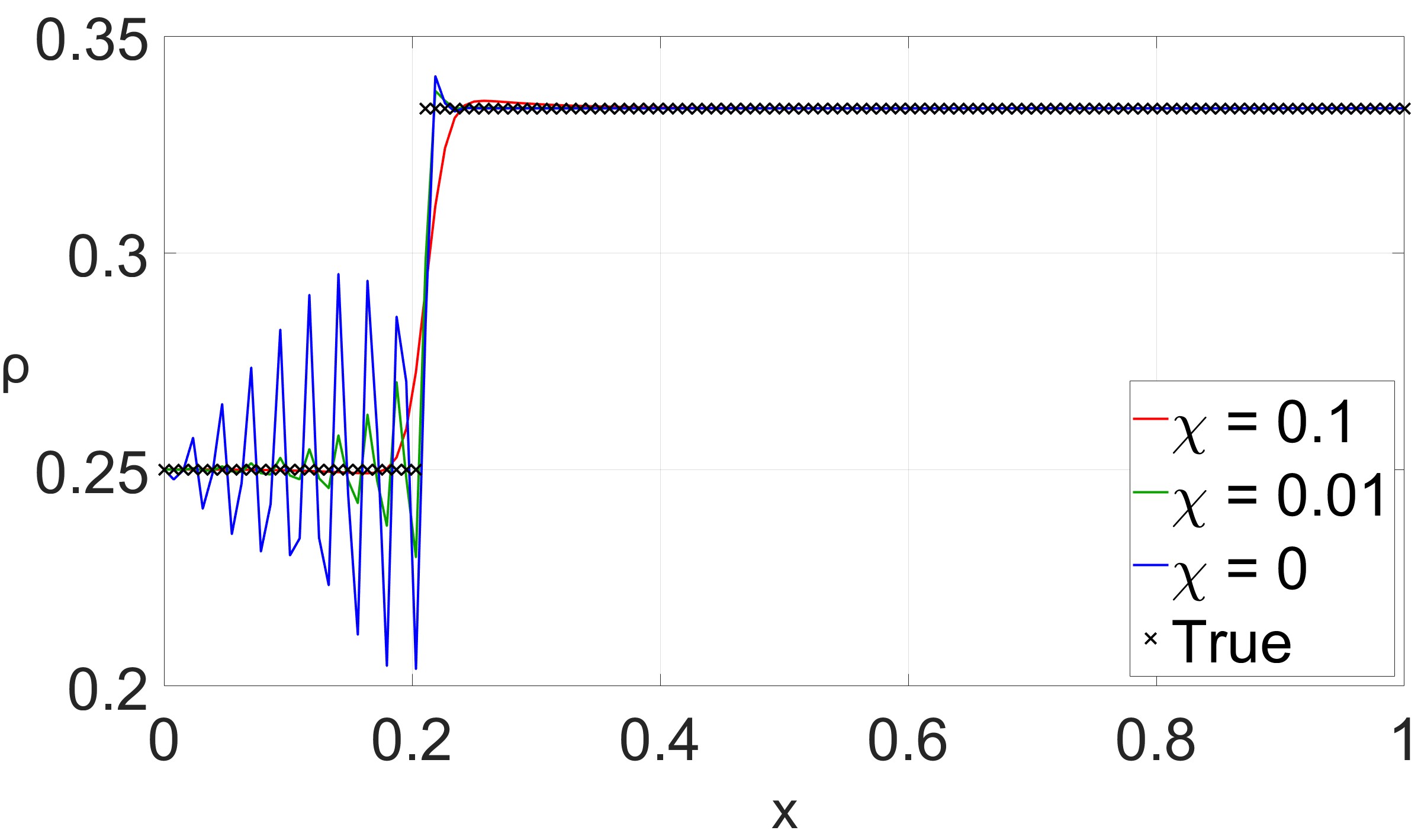}
                    \caption{Time $t=0.5, N=0$ with $P1$ elements}
                    \label{fig:ShockWaveHalfTimeN0P1}
                \end{subfigure}
                \begin{subfigure}{.5\textwidth}
                    \includegraphics[clip,scale=0.08]{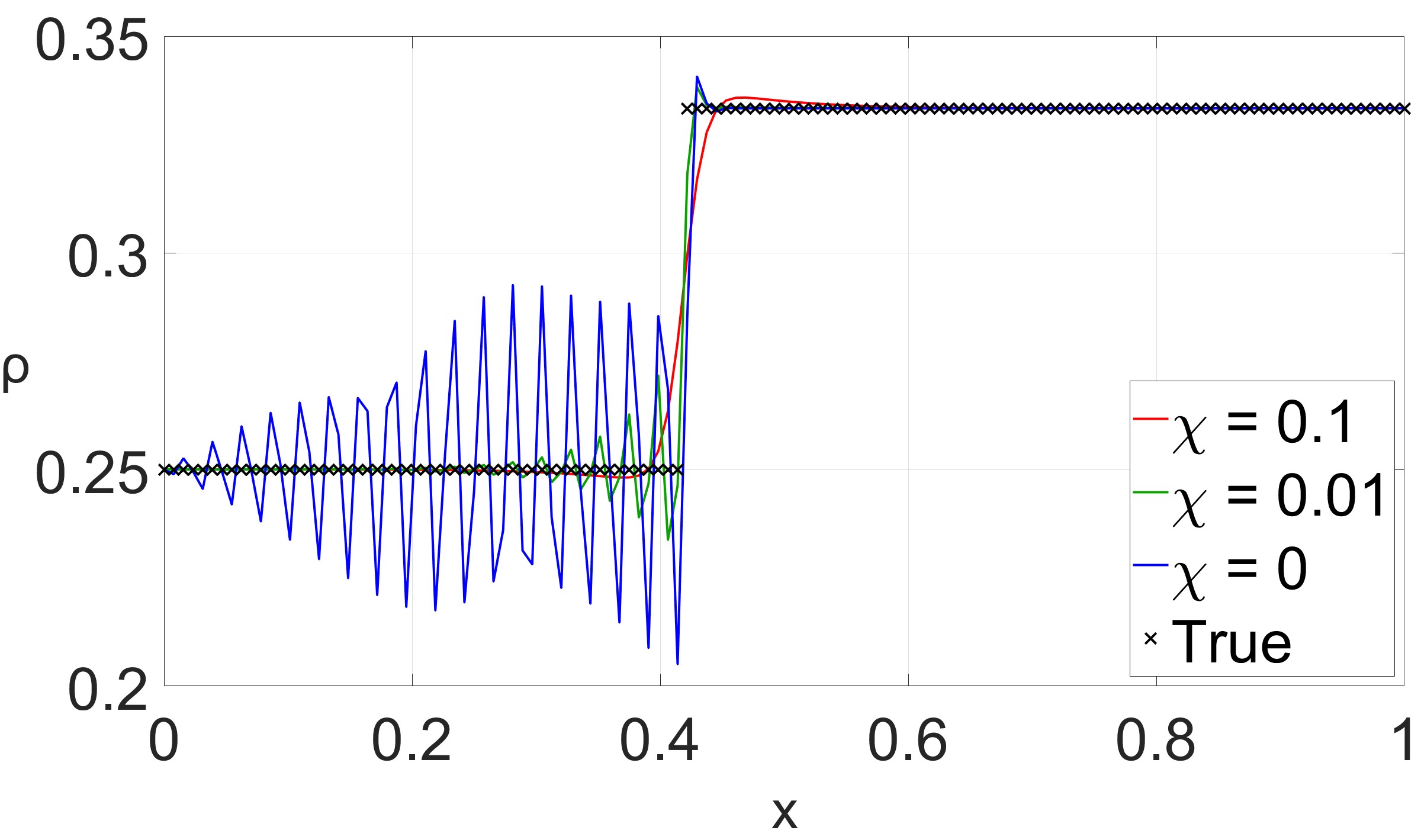}
                    \caption{Time $t=1, N=0$ with $P1$ elements}
                    \label{fig:ShockWaveFullTimeN0P1}
                \end{subfigure}
                \begin{subfigure}{.5\textwidth}
                    \includegraphics[clip,scale=0.08]{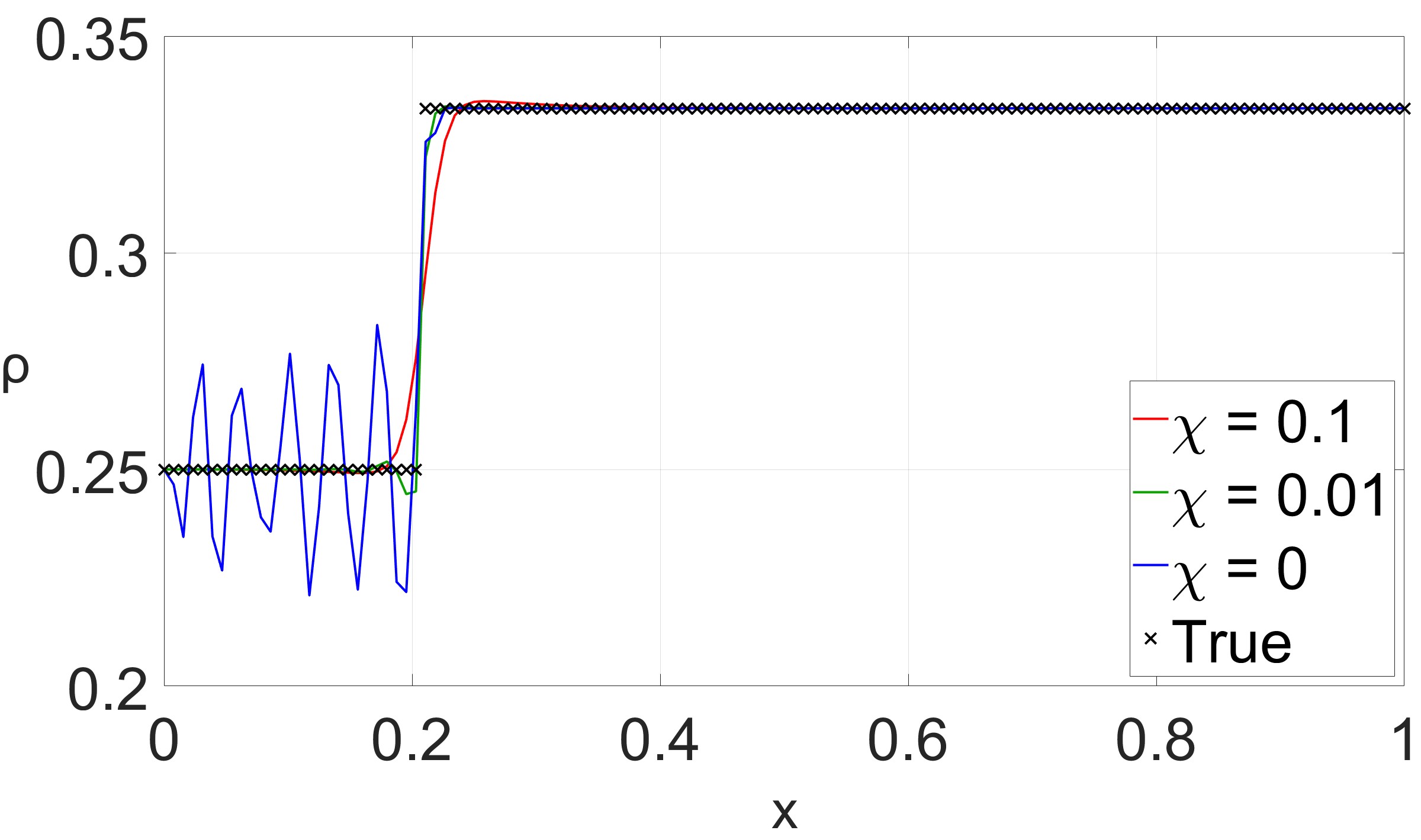}
                    \caption{Time $t=0.5, N=0$ with $P2$ elements}
                    \label{fig:ShockWaveHalfTimeN0P2}
                \end{subfigure}
                \begin{subfigure}{.5\textwidth}
                    \includegraphics[clip,scale=0.08]{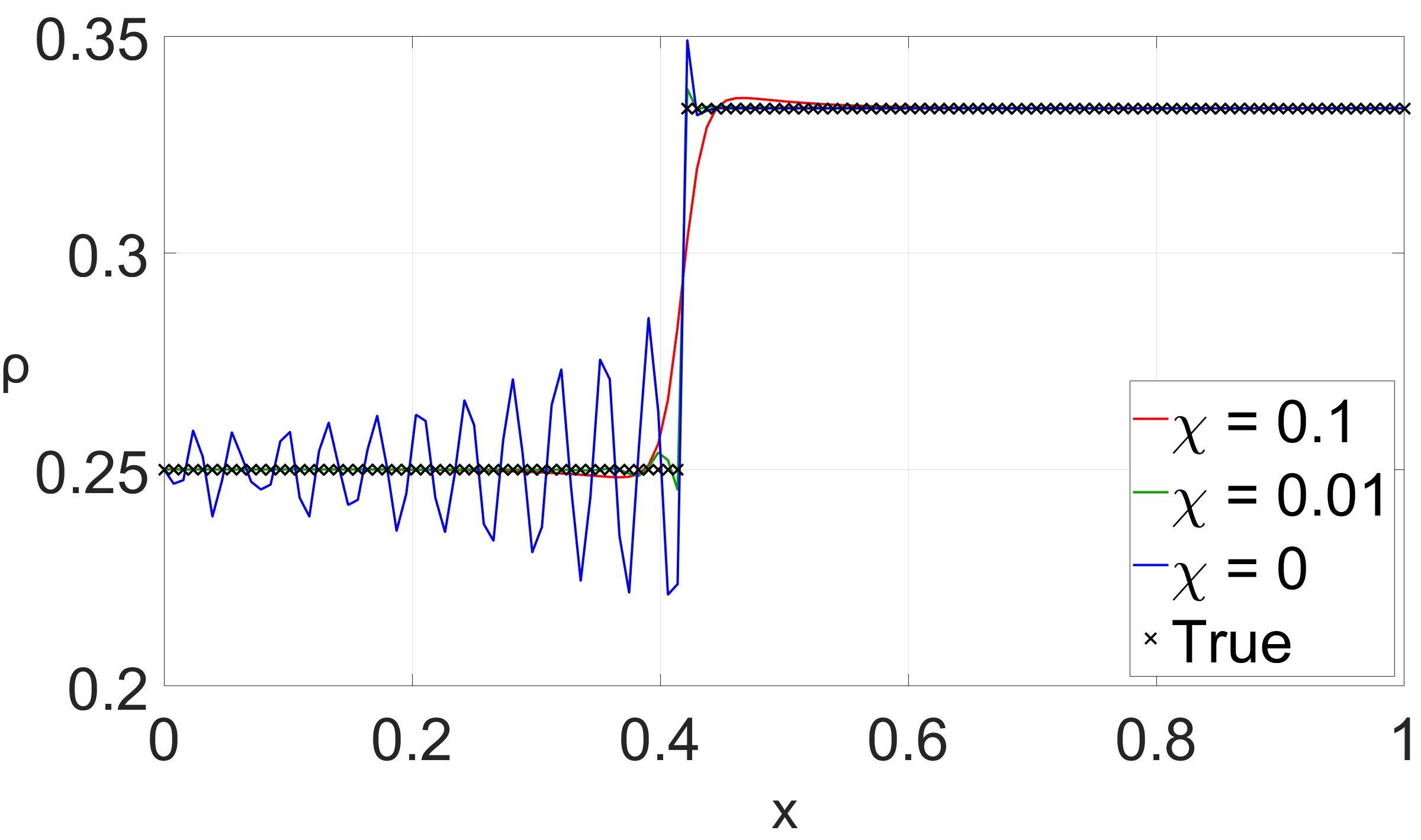}
                    \caption{Time $t=1, N=0$ with $P2$ elements}
                    \label{fig:ShockWaveFullTimeN0P2}
                \end{subfigure}
                \begin{subfigure}{.5\textwidth}
                    \includegraphics[clip,scale=0.08]{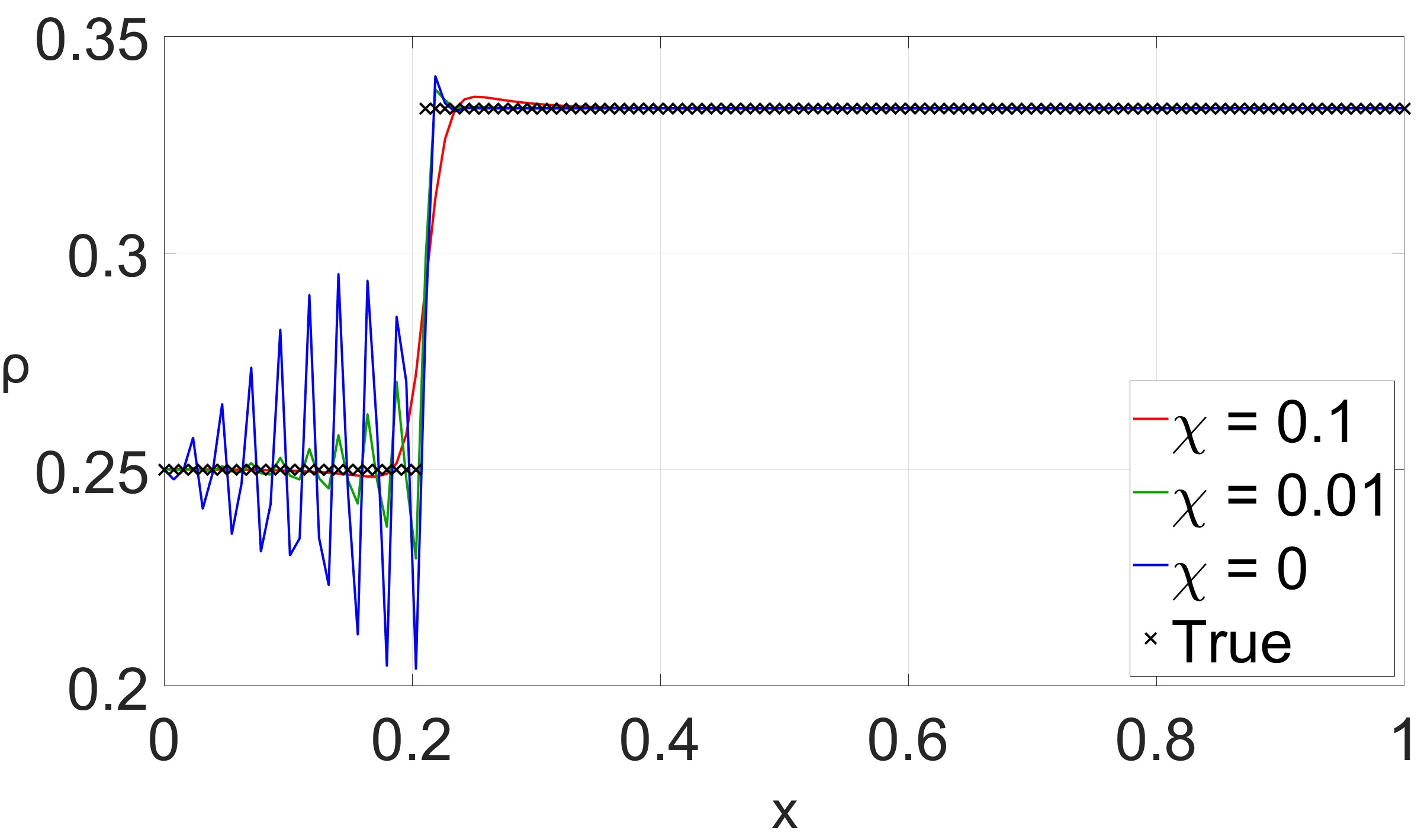}
                    \caption{Time $t=0.5, N=1$ with $P1$ elements}
                    \label{fig:ShockWaveHalfTimeN1P1}
                \end{subfigure}
                \begin{subfigure}{.5\textwidth}
                    \includegraphics[clip,scale=0.08]{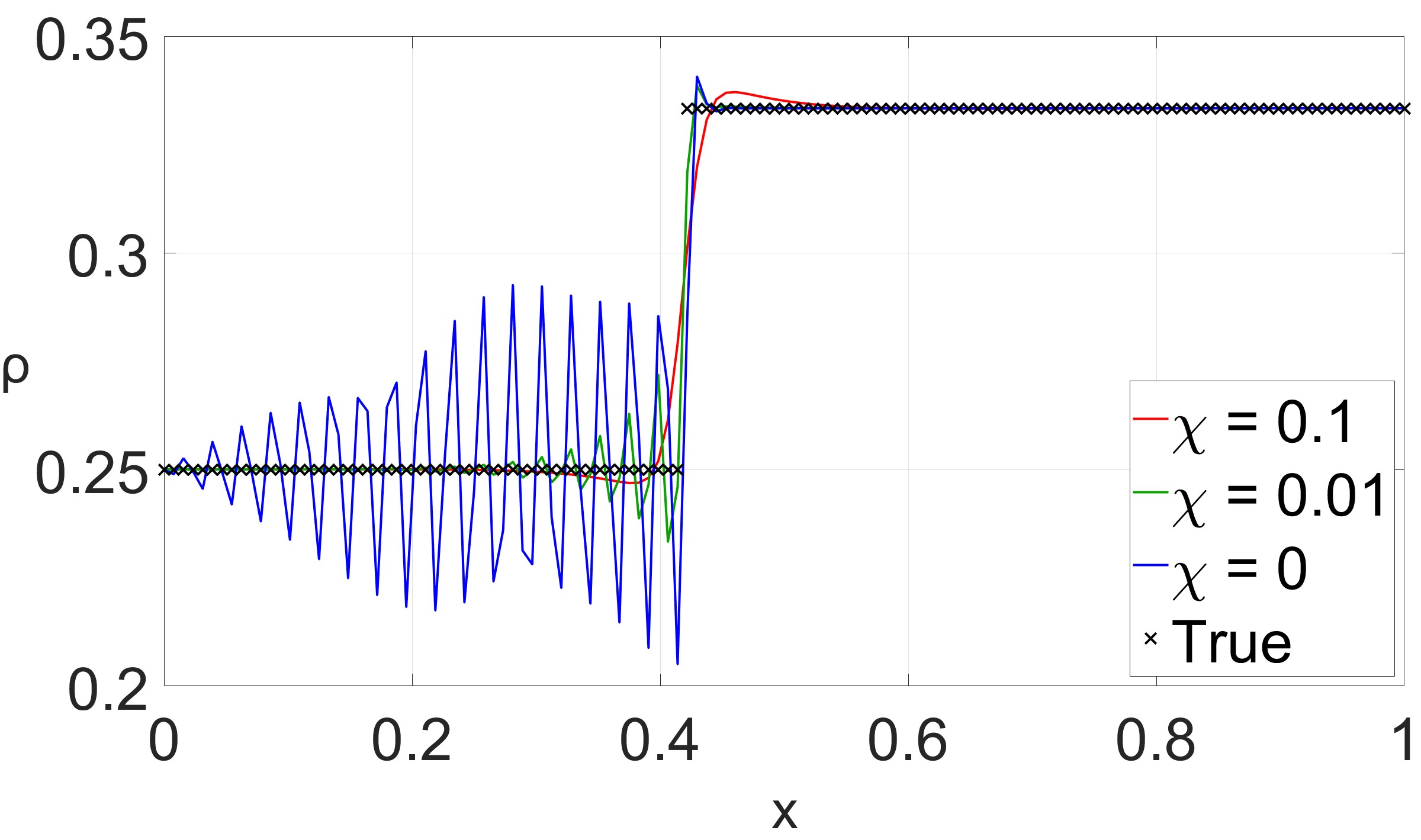}
                    \caption{Time $t=1, N=1$ with $P1$ elements}
                    \label{fig:ShockWaveFullTimeN1P1}
                \end{subfigure}
                \begin{subfigure}{.5\textwidth}
                    \includegraphics[clip,scale=0.08]{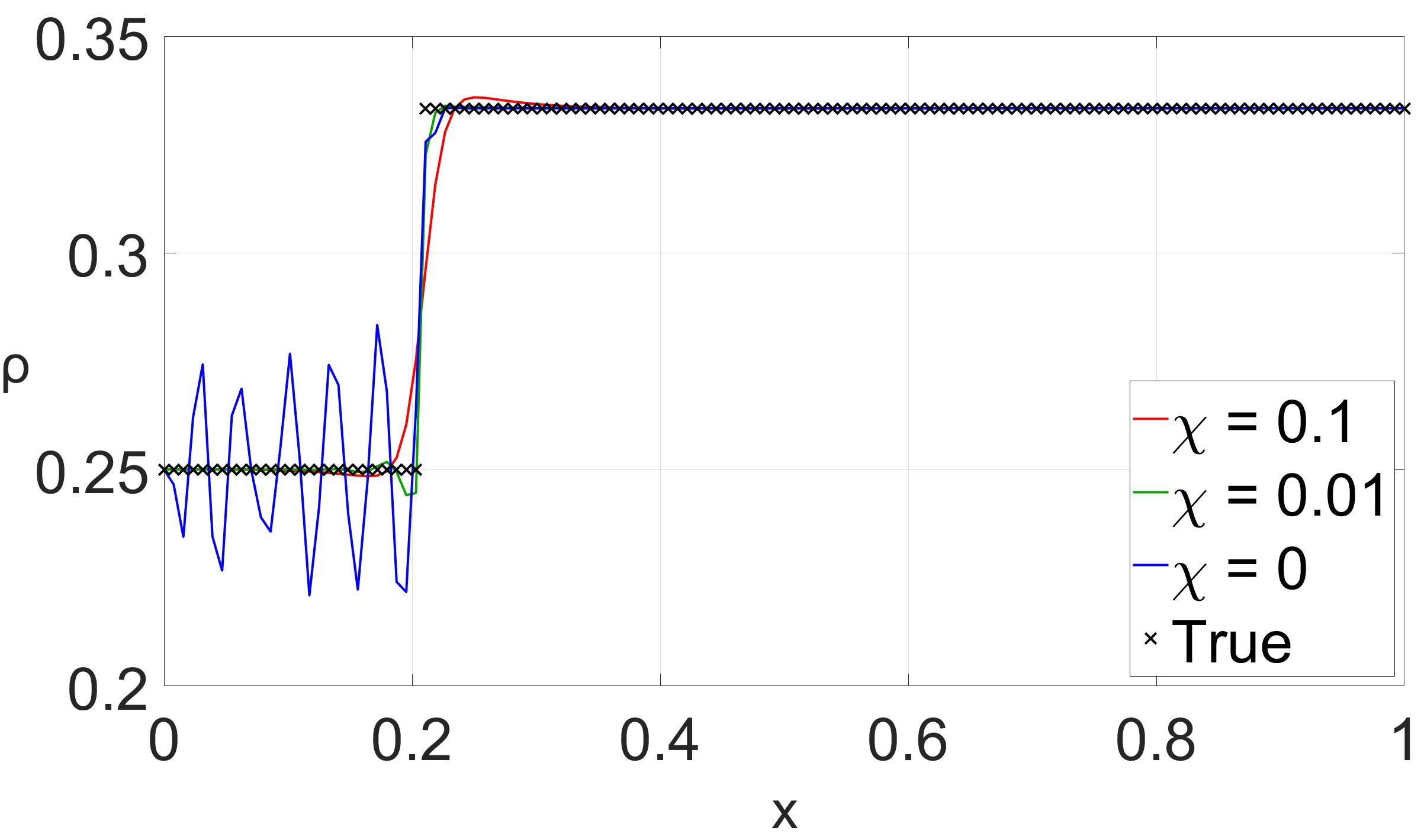}
                    \caption{Time $t=0.5, N=1$ with $P2$ elements}
                    \label{fig:ShockWaveHalfTimeN1P2}
                \end{subfigure}
                \begin{subfigure}{.5\textwidth}
                    \includegraphics[clip,scale=0.08]{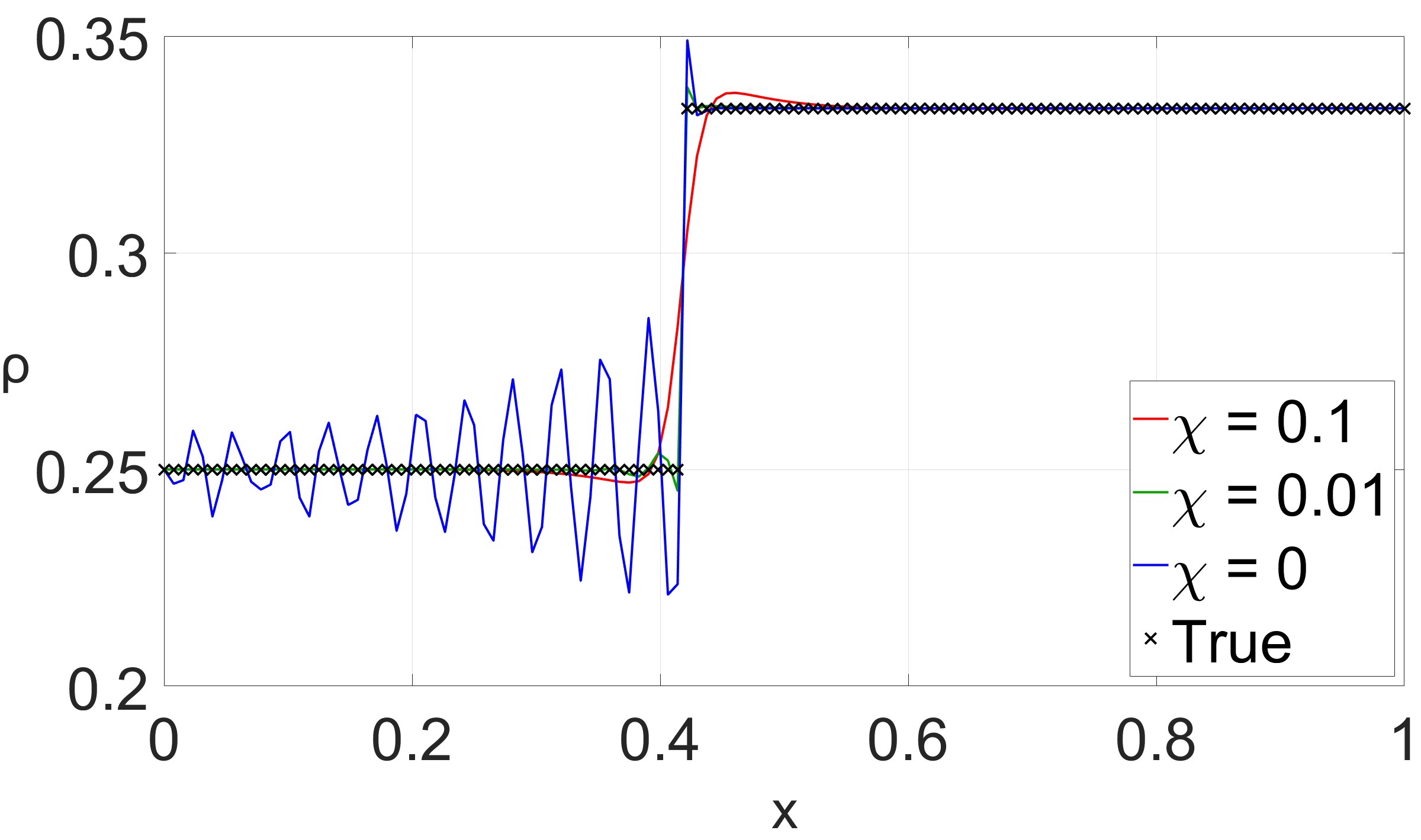}
                    \caption{Time $t=1, N=1$ with $P2$ elements}
                    \label{fig:ShockWaveFullTimeN1P2}
                \end{subfigure}
                \caption{Shock wave density profiles for various values of $\chi$}
            \end{figure}

            \begin{figure}[!htb]
                \begin{subfigure}{.5\textwidth}
                    \centering
                    \includegraphics[clip,scale=0.08]{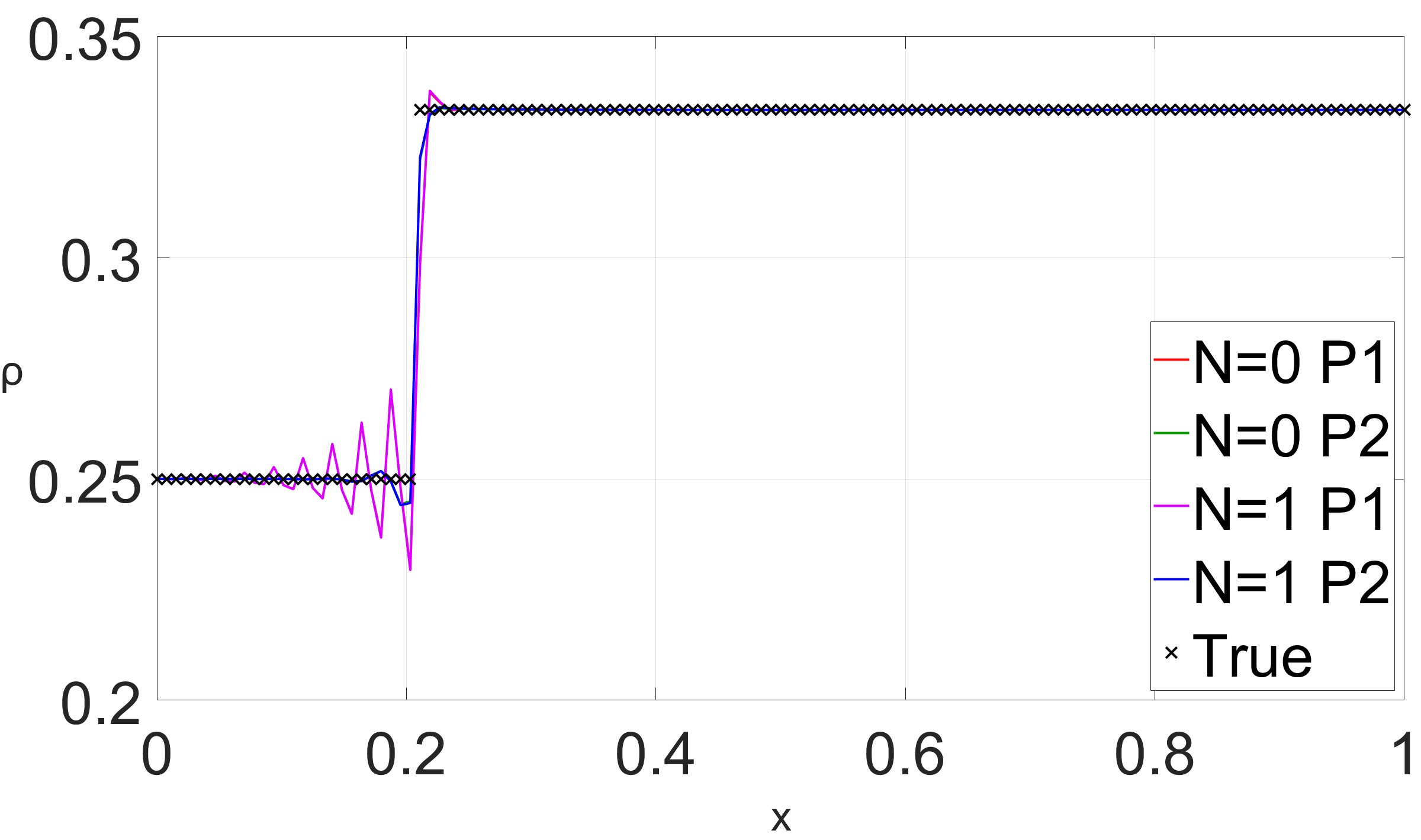}
                    \caption{$t=0.5, \chi=0.01$}
                    \label{fig:ShockwaveHalfTimeChi0.01}
                \end{subfigure}
                \begin{subfigure}{.5\textwidth}
                    \centering
                    \includegraphics[clip,scale=0.08]{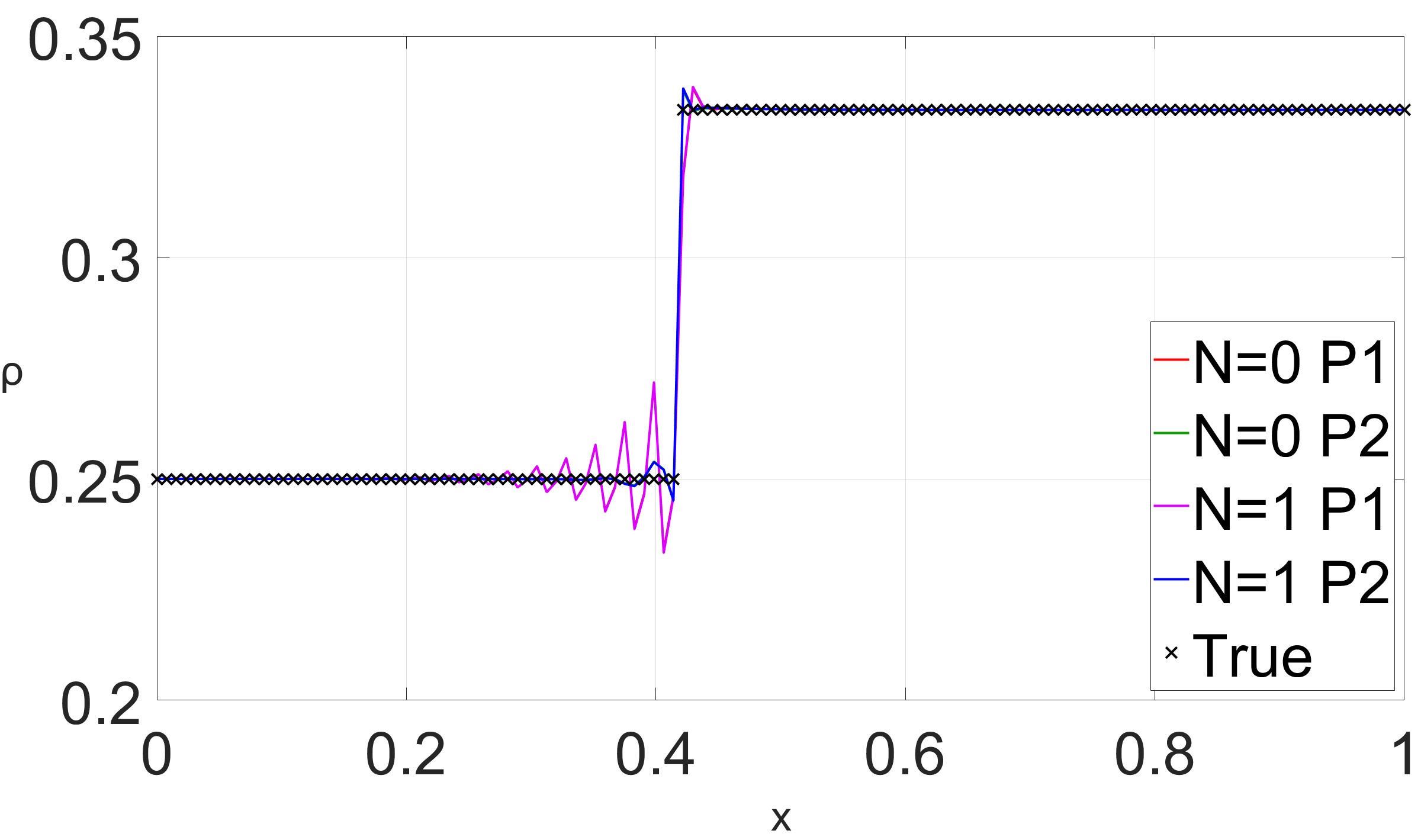}
                    \caption{$t=1, \chi=0.01$}
                    \label{fig:ShockwaveFullTimeChi0.01}
                \end{subfigure}
                \caption{Shock wave density profiles for various combinations of $N,k$}
            \end{figure}
            
\section{Conclusion}
\label{sec: end}
    After deriving numerical algorithms for the LWR model with Greenshield's velocity, stability and convergence results were presented. Computational simulations supported these results while delivering insights on the behavior of solutions to these algorithms. For instance, the nature of shockwave solutions were captured with Vreman stablization, while omitting stabilization resulted in spurious oscillations.
    
    It has been shown that the spatial and temporal rates of convergence can be improved with the use of Vreman stabilization and time filtering, respectively. Additionally, we get a better look at the dependence of solutions to Algorithms \ref{alg: BEGM} and \ref{alg: FDTFGM} on the order of deconvolution, degree of finite element space, amount of stabilization, and other numerical parameters. Thus, coupling the model with Vreman stabilization and time filtering can serve as a robust model for density profiles of RNA polymerase transcription, which can yield implications about protein synthesis.

\section{Acknowledgements}
The contribution of authors, Dr. Davis and Dr. Pahlevani, was supported by the National Science Foundation under Awards DMS-1951510 and DMS-1951563. Any opinions, findings, and conclusions or recommendations expressed in this material are those of
the author(s) and do not necessarily reflect the views
of the National Science Foundation.

\biboptions{sort&compress}
\bibliographystyle{elsarticle-num-names} 
\bibliography{cas-refs}





\end{document}